\documentclass{siamltex}
\usepackage{amsfonts}
\usepackage{amsmath}
\usepackage[all]{xy}

\usepackage{tikz}

\usepackage{graphicx}

  \title{Convergence in strongly monotone systems with an increasing first integral}

\author{Murad Banaji\footnotemark[1] \footnotemark[4]
\and David Angeli\footnotemark[2]}

\begin{document}

\maketitle

\renewcommand{\thefootnote}{\fnsymbol{footnote}}

\footnotetext[1]{Dept. of Mathematics, University College London, UK, and Dept. of Biological Sciences, University of Essex, Colchester, UK. }
\footnotetext[2]{Dept. of Electrical and Electronic Engineering, Imperial College London, UK, and Dip. Sistemi e Informatica, University of Florence, Italy.}
\footnotetext[4]{m.banaji@ucl.ac.uk. Research funded by EPSRC grant EP/D060982/1.}

\renewcommand{\thefootnote}{\arabic{footnote}}

\begin{abstract}
In this paper we generalise a useful result due to J. Mierczy\'nski which states that for a strictly cooperative system on the positive orthant, with increasing first integral, all bounded orbits are convergent. Moreover any equilibrium attracts its entire level set, and there can be no more than one equilibrium on any level set. Here, more general state spaces and more general orderings are considered. Let $Y \subset K \subset \mathbb{R}^n$ be any two proper cones. Given a local semiflow $\phi$ on $Y$ which is strongly monotone with respect to $K$, and which preserves a $K$-increasing first integral, we show that every bounded orbit converges. Again, each equilibrium attracts its entire level set, and there can be no more than one equilibrium on any level set. An application from chemical dynamics is provided.
\end{abstract}



\section{Introduction}
The study of the qualitative behaviour of dynamical systems is a vast subject with applications in many fields. In particular monotone systems, i.e. systems which preserve some partial order on the state space, have been intensively studied, with a range of qualitative results on asymptotic behaviour in these systems. See \cite{hirschsmithalt} for a recent survey or \cite{halsmith} for an earlier monograph on the subject. When the state space is some subset of Euclidean space, and the preserved partial order is the ``natural'' order generated by the positive orthant, we get so-called cooperative systems. The fundamental notions connected with cooperativity extend to more general orderings (\cite{Walcher} for example). 

Monotonicity constrains the behaviour of dynamical systems, for example ruling out attracting nontrivial periodic orbits, provided at least one point of any periodic orbit it accessible from above or below \cite{hirschsmithalt}. When a dynamical system is strongly monotone (to be defined below) behaviour is constrained further: for almost all initial conditions bounded solutions converge to the set of equilibria, a result initially proved for strongly cooperative systems by M. Hirsch in \cite{hirsch85}. Sometimes generic convergence claims can be strengthened, provided additional structure is available. For instance, global convergence (i.e. convergence of {\em every} bounded orbit) can be obtained in a variety of special cases: for tridiagonal strongly cooperative systems \cite{smillie}; when a system enjoys so-called ``positive translation invariance'' \cite{as}; and when a strongly cooperative system is endowed with a strictly increasing first integral (the result of Mierczy\'nski \cite{mierczynski} to be generalised here). In this latter case, the conclusions are stronger still: there can be no more than one equilibrium on each level set of the first integral, and when it exists, such an equilibrium attracts the whole level set. In the same spirit is Theorem 5 of \cite{leenheer}, which shows how for lattice state spaces, and provided a unique equilibrium exists, all bounded solutions converge to this equilibrium.  

The importance of Mierczy\'nski's result stems from the fact that in a variety of applications natural constraints lead to order preservation, while at the same time conservation laws define preserved functions. However, as shown for chemical reaction networks in \cite{banajidynsys,angelileenheersontag}, the preserved partial orders may not be induced by orthants, and indeed, may not be induced by simplicial cones. Thus appropriate generalisations of Mierczy\'nski's result potentially have useful application in these areas. A small example of such an application will be presented later.\\

\section{The main result}
We state the main result and outline the proof.

\vspace{0.25cm}
\begin{definition}
A {\bf proper cone} in $\mathbb{R}^n$ will be defined as a closed, convex, pointed cone with nonempty interior \cite{berman}. 
\end{definition}
\vspace{0.25cm}

Let $Y, K$ be proper cones in $\mathbb{R}^n$ with $K \supset Y$. From now on, all inequalities are with respect to the ordering defined by $K$, i.e. $x \leq y$ will mean $y - x \in K$, $x < y$ will mean $x \leq y$ and $x \not = y$, $x \ll y$ will mean $y - x \in \mathrm{int}(K)$, etc. Define $K^{*}$ to be dual cone to $K$, i.e. $K^* = \{y \in \mathbb{R}^n\,|\, \langle y, k \rangle \geq 0\,\,\mbox{for all}\,\, k \in K\}$. Consider a system 
\begin{equation}
\label{eq1}
\dot x = F(x)
\end{equation}
on $Y$, where $F(x)$ is locally Lipschitz and so defines a local semiflow $\phi$ on $Y$. Assume that:
\begin{enumerate}
\item $\phi$ is {\bf strongly monotone} with respect to $K$, i.e. $x > y \Rightarrow \phi_t(x) \gg \phi_t(y)$ for all $t > 0$ such that $\phi_t(x)$ and $\phi_t(y)$ are defined).
\item The system has a $C^1$ {\bf first integral} $H:Y \to \mathbb{R}$, such that for each $y \in Y$, i)~$\langle \nabla H(y), F(y) \rangle = 0$ and ii) $\nabla H(y) \in \mathrm{int}(K^*)$. 
\end{enumerate}

{\bf Remarks.} Since $K$ is proper, $K^*$ is automatically a proper cone \cite{berman}, and hence has nonempty interior. From here on if we refer to $\phi_t(y)$, the assumption is that $t$ is in the interval of existence of the solution taking initial value $y$. 

For convenience, and without loss of generality, we assume $H(0) = 0$. Given $x, y \in Y$ with $y > x$, by convexity of $Y$, the line segment between $x$ and $y$ lies in $Y$, and by integrating $\nabla H$ along this line segment, we get $H(y) > H(x)$. 
This implies that the level sets of $H$ are unordered, and since $H(0) = 0$, $H(y) > 0$ for all $y \in Y\backslash\{0\}$. Denote by $M$ the l.u.b. of the values of $H$, so that $0 < M \leq \infty$. Given any $y \in Y$, there exists $z \in Y$, $z > y$, so $M > H(y)$. As a continuous scalar function on a convex (and hence connected) set, $H:Y \to [0, M)$ is surjective.

\vspace{0.25cm}
\begin{definition}
$S(0,h) \equiv \{x \in Y\,|\, H(x) = h\}$ is the {\bf level set} associated with $h \in [0, M)$.
\end{definition}
\vspace{0.25cm}

\begin{definition}
The {\bf equilibrium set} $E$ is defined as $E \equiv \{y \in Y\,| F(y) = 0\}$.
\end{definition}
\vspace{0.25cm}

Note that $S(0,h)$ is always closed, but may be unbounded. $\langle \nabla H(y), F(y) \rangle = 0$ implies that each $S(0,h)$ is forward invariant under $\phi$. By continuity of $F$, $E$ is closed. The main result of this paper is the following:

\vspace{0.25cm}
\begin{theorem}
\label{mainthm}
There exists some $0 < M^{'} \leq M$ such that
\begin{enumerate}
\item For each $h \in [0, M^{'})$, $S(0,h)$ contains a unique equilibrium to which each orbit on $S(0,h)$ converges.
\item If $M^{'} \not = M$, then for each $h \in [M^{'}, M)$, $S(0, h)$ contains no equilibria, and every orbit on $S(0,h)$ is unbounded. 
\end{enumerate}
\end{theorem}
\vspace{0.25cm}

{\bf Remark.} The result tells us not only that every bounded orbit of (\ref{eq1}) converges to an equilibrium, but also rules out multiple equilibria on any level set. 

{\bf An immediate corollary.} By insisting that $K \supset Y$, the result is apparently phrased in less generality than possible. However a more general result follows immediately. Consider the case where $Y$ is any forward invariant subset of $\mathbb{R}^n$ containing an equilibrium which (without loss of generality) we take to be at the origin. 
Assume that $Y^{'} \equiv Y \cap K$ is a nonempty, closed, convex, pointed cone: for example, $Y$ may be $\mathbb{R}^n$, in which case $Y^{'} = K$, or $Y$ may be any other closed cone, not necessarily convex, which intersects $K$. $Y^{'}$ has nonempty interior in $\mathrm{Aff}(Y^{'})$, the smallest affine subspace containing $Y^{'}$, and by easy arguments (Lemma~\ref{forwardinvar} below) it is forward invariant. Replacing $Y$ with $Y^{'}$, $\mathbb{R}^n$ with $\mathrm{Aff}(Y^{'})$, and $K$ with $K \cap \mathrm{Aff}(Y^{'})$ (which is proper in $\mathrm{Aff}(Y^{'})$), Theorem~\ref{mainthm} can immediately be applied to get global convergence on each set $S(0,h) \cap Y^{'}$ which contains an equilibrium. 

{\bf Summary of the arguments}. The proof of Theorem~\ref{mainthm} will be presented after preliminary results. The fundamental geometrical ideas are closely related to those in \cite{mierczynski}. The greater generality however presents some technical difficulties -- for example the fact that $Y$ is not necessarily a lattice under the order induced by $K$ makes observations which would be immediate, such as that the equilibrium set is ordered, harder to prove. 

Ultimately, as in \cite{mierczynski}, we will define a continuous scalar function $L$ which increases strictly along all orbits except equilibria. Given any point $y \in Y$, we will show that the set $y - \partial K$ intersects $E$ at a unique point, $Q(y)$. Uniqueness of $Q(y)$ will follow from the fact that $E$ is embedded in $Y$ in a rather special way: $E$ is totally ordered and homeomorphic to a half-open line segment. $L$ is then defined by $L(y) = H(Q(y))$. It will not be hard to show that the assumption of strong monotonicity implies that $L$ is increasing at any nonequilibrium point. 

\section{Preliminaries}
\label{secprelim}

{\bf Notation.} Given any set $X \subset \mathbb{R}^n$ the smallest affine subspace of $\mathbb{R}^n$ containing $X$ will be termed $\mathrm{Aff}(X)$. We will refer to the relative interior of $X$ and the relative boundary of $X$ with respect to $\mathrm{Aff}(X)$ as $\mathrm{ri}(X)$ and $\mathrm{relbd}(X)$ respectively. If we refer to the relative interior of $X$ with respect to a set other than $\mathrm{Aff}(X)$, then this will be made clear.

\vspace{0.25cm}

Note the following basic properties of convex sets \cite{nikaido,websterconvexity}:
\begin{enumerate}
\item[{\bf O1}] Given a proper cone $K \subset \mathbb{R}^n$, some $p \in \mathrm{int}(K), y \in K$, then $p + y \in \mathrm{int}(K)$.
\item[{\bf O2}] Given a closed, convex set $X \subset \mathbb{R}^n$, some $p \in X$, and any $y \in \mathbb{R}^n$, the ray $\{p+ty\,|\, t \geq 0\}$ either lies in $X$ or there exists $t^{'} \geq 0$ such that $p+ty \in X$ for $0 \leq t \leq t^{'}$ and $p+ty \not \in X$ for $t > t^{'}$. We say that the ray ``exits $X$'' at $p+t^{'}y$. If $p \in \mathrm{ri}(X)$, and $y \in \mathrm{Aff}(X)$, then, provided it exists, $t^{'} > 0$, and $p+t^{'}y$ is the unique point in $\{p+ty\,|\, t \geq 0 \} \cap \mathrm{relbd}(X)$.
\end{enumerate}
\vspace{0.25cm}

\begin{definition}
A {\bf $k$-dimensional ball} in $\mathbb{R}^n$ will be defined as any set homeomorphic to a nonempty, compact, convex, set $X \subset \mathbb{R}^n$ such that $\mathrm{Aff}(X)$ has dimension $k$. We allow $k = 0$, i.e. a $0$-dimensional ball is a single point.
\end{definition}
\vspace{0.25cm}


\begin{lemma}
\label{upperboundlem}
Given any bounded set $X \subset Y$, there exists $z \in Y$, $z > X$. 
\end{lemma}
\begin{proof}
Let $d = \sup_{x \in X}|x|$. Choose any $z^{'} \in \mathrm{int}(Y)$. Since $\partial Y$ is closed, $d_{min} \equiv \inf_{y \in \partial Y}|z^{'}-y| > 0$. Choose $t > d/d_{min}$, so that for any $x \in X$, $|x/t| < d_{min}$, implying $z^{'} - x/t \in \mathrm{int}(K)$. Define $z = tz^{'}$. It is immediate that $z \in \mathrm{int}(Y)$, and moreover $z-x = tz^{'} - x = t(z^{'} - x/t) \in \mathrm{int}(K)$, so $z > X$. 
\end{proof}\\

\begin{lemma}
For all $y \in Y$, $F(y) \not \in K \backslash\{0\}$. Hence $F(0) = 0$. 
\end{lemma}

\begin{proof}
$\nabla H(y) \in \mathrm{int}(K^*)$ implies that $\langle \nabla H(y), v \rangle > 0$ for all $v \in K\backslash\{0\}$. Consequently, $\langle \nabla H(y), F(y) \rangle = 0$ implies that either $F(y) = 0$ or $F(y) \not \in K$. As a corollary, $F(0) = 0$, since otherwise $F(0) \not \in Y$ 
violating invariance of $Y$.
\end{proof}\\

We now choose some arbitrary but fixed unit vector $g \in \mathrm{int}(K^*)$. Throughout the rest of this paper, $g$ will refer to this vector. To simplify the arguments to follow, some notation is collected in Table~\ref{tableofnotation} below. 

\begin{table}[h]
\begin{tabular}{|p{0.46\textwidth}p{0.46\textwidth}|}
\hline
$c^{+} \equiv Y \cap (c+K)$ & $c^{-} \equiv Y \cap (c-K)$\\
$c^{\partial} \equiv Y \cap ((c + \partial K) \cup (c - \partial K))$ & $P(g, r)\equiv \{y \in \mathbb{R}^n\,|\,\langle g, y\rangle = r\}$\\
$\Delta_{+}(g, r) \equiv \{y\in \mathbb{R}^n\,|\,\langle g, y\rangle \geq r\}$ & $\Delta_{-}(g, r) \equiv \{y\in \mathbb{R}^n\,|\,\langle g, y\rangle \leq r\}$\\
$P_{+}(c, g, r) \equiv P(g, r) \cap c^{+}$ & $P_{-}(c, g, r) \equiv P(g, r) \cap c^{-}$\\
$\Delta_{+}(c,g, r) \equiv \Delta_{-}(g, r) \cap c^{+}$ & $\Delta_{-}(c,g, r) \equiv \Delta_{+}(g, r) \cap c^{-}$\\
$S_{+}(c,h) \equiv \{y \in c^{+}\,|\,H(y) = h\}$ & $S_{-}(c,h) \equiv \{y \in c^{-}\,|\,H(y) = h\}$\\
$D_{+}(c,h) \equiv \{y \in c^{+}\,|\,H(y) \leq h\}$ & $D_{-}(c,h) \equiv \{y \in c^{-}\,|\,H(y) \geq h\}$\\
$P(c, g, r) \equiv \left \{ \begin{array}{ll}P_{+}(c, g, r), & r \geq \langle g, c \rangle\\P_{-}(c, g, r), & r < \langle g, c \rangle\end{array}\!\!\right.$ & $\Delta(c, g, r) \equiv \left \{ \begin{array}{ll}\Delta_{+}(c, g, r), & \!r \geq \langle g, c \rangle\\ \Delta_{-}(c, g, r), & \!r < \langle g, c \rangle \end{array}\!\!\!\right.$\\
$S(c, h) \equiv \left\{ \begin{array}{ll}S_{+}(c, h), & h \geq H(c)\\S_{-}(c, h), & h < H(c)\end{array}\right.$ & $D(c, h) \equiv \left\{ \begin{array}{ll}D_{+}(c, h), & h \geq H(c)\\D_{-}(c, h), & h < H(c)\end{array}\right.$\\
$r_{c\to x}  \equiv \{c + tx\,|\,t \geq 0\}$ & $[c,x] \equiv \{(1-\lambda) c + \lambda x\,|\, \lambda \in [0,1]\}$\\
\hline
\end{tabular}
\caption{\label{tableofnotation}Some notation. $c, x$ are any vectors in $Y$, $r \in [0, \infty)$ and $h \in [0, M)$. $c^{+}$ is the set of points in $Y$ which are greater than or equal to $c$, while $c^{-}$ is analogously defined. $c^{+}, c^{-}$ are closed and convex, as the intersection of closed, convex sets. $P(g, r)$ is an $n-1$ dimensional unordered hyperplane, parallel to $g^\perp$. $\Delta_{+}(g, r)$ and $\Delta_{-}(g, r)$ are the associated half-spaces. $\Delta_{+}(c, g, r)$ is the area of $c^{+}$ bounded above by $P(g, r)$, while $\Delta_{-}(c,g, r)$ is the area of $c^{-}$ bounded below by $P(g, r)$. Assuming $x \not = c$, $r_{c \to x}$ is the ray originating at $c$ and passing through $x$ (which we will refer to as a ``nontrivial ray''), and $[c,x]$ is the closed line segment connecting $c$ and $x$. Justification for the definitions of $P(c, g, r), \Delta(c, g, r), S(c, h)$ and $D(c, h)$ is presented in Lemmas~\ref{notate1}~and~\ref{notate2}.}
\end{table}

\vspace{0.25cm}
\begin{lemma}
\label{nobordereq}
Let $c \in E$. Then there are no equilibria in $c^{\partial}\backslash\{c\}$, and in particular, $E \subset \{0 \} \cup (Y \backslash \partial K)$.
\end{lemma}
\begin{proof}
Assume the contrary and consider an equilibrium $c_1 \in c+\partial K$ ($c_1 \not = c$). Then $\phi_t(c_1) - \phi_t(c) = c_1 - c \in \partial K$ for $t > 0$. But by the assumption of strong monotonicity, $\phi_t(c_1) - \phi_t(c) \in \mathrm{int}(K)$ for $t > 0$, a contradiction. The argument is similar if $c_1 \in c-\partial K$. Since $0$ is an equilibrium, all other equilibria lie in $Y \backslash \partial K$.
\end{proof}\\

\begin{lemma}
\label{forwardinvar}
Let $c \in E$. Then $c^{+}$ and $c^{-}$ are forward invariant.
\end{lemma}
\begin{proof}
If $y \geq c$, then by monotonicity $\phi_t(y) \geq \phi_t(c) = c$ for all $t \geq 0$ -- i.e. if $y \in c+K$, then $\phi_t(y) \in c+K$. A similar argument shows that $c-K$ is forward invariant. As $Y$ is invariant by assumption, $c^{+}$ and $c^{-}$ are the intersection of forward invariant sets and are hence forward invariant. 
\end{proof}\\

Lemmas~\ref{notate1}~and~\ref{notate2} below clarify the definitions of $P(c, g, r)$, $\Delta(c, g, r)$, $S(c, h)$ and $D(c, h)$ in Table~\ref{tableofnotation}. The situations are represented schematically in Figure~\ref{basic0}.

\vspace{0.25cm}
\begin{lemma}
\label{notate1}
Given any $c \in Y$ and $r \geq 0$:
\begin{enumerate}
\item If $r > \langle g, c\rangle$, then $P_{-}(c, g, r), \Delta_{-}(c,g, r)$ are empty and $P_{+}(c, g, r), \Delta_{+}(c,g, r)$ are nonempty. 
\item If $r < \langle g, c\rangle$, then $P_{+}(c, g, r)$, $\Delta_{+}(c,g, r)$ are empty and $P_{-}(c, g, r)$, $\Delta_{-}(c,g, r)$ are nonempty. 
\item If $r = \langle g, c\rangle$, then  $P_{-}(c, g, r) = P_{+}(c, g, r) = \Delta_{-}(c, g, r) = \Delta_{+}(c, g, r) = \{c\}$.
\end{enumerate}
\end{lemma}
\begin{proof}
For any $x \in K\backslash\{0\}$, $\langle g, x \rangle  > 0$, and so $\langle g, c + x \rangle > \langle g, c \rangle$ and  $\langle g, c - x \rangle < \langle g, c \rangle$. It follows immediately that if $r > \langle g, c\rangle$, then $P_{-}(c, g, r), \Delta_{-}(c,g, r)$ are empty, if $0 \leq r < \langle g, c\rangle$, then $P_{+}(c, g, r), \Delta_{+}(c,g, r)$ are empty, and if $r = \langle g, c\rangle$, then  $P_{-}(c, g, r) = P_{+}(c, g, r) = \Delta_{-}(c, g, r) = \Delta_{+}(c, g, r) = \{c\}$. For any $y \in Y\backslash\{0\}$ and $r > 0$, $ry/\langle g, y \rangle \in P_{+}(0, g, r)$. For $c \not = 0$, define $t = r/\langle g, c \rangle$. It is easy to check that if $r > \langle g, c\rangle$, then $tc \in P_{+}(c, g, r)$, and if $0 \leq r < \langle g, c\rangle$, then $tc \in P_{-}(c, g, r)$, proving the remaining claims. 
\end{proof}

\vspace{0.25cm}
\begin{lemma}
\label{notate2}
Given any $c \in Y$, $h \in [0, M)$, 
\begin{enumerate}
\item If $h > H(c)$, then $S_{-}(c,h), D_{-}(c,h)$ are empty.
\item If $h < H(c)$, then $S_{+}(c,h), D_{+}(c,h)$ are empty.
\item If $h = H(c)$, then $S_{-}(c,h) =  D_{-}(c,h) = S_{+}(c,h) = D_{+}(c,h) = \{c\}$.
\end{enumerate}

\end{lemma}
\begin{proof}
For any $x \in K\backslash\{0\}$, where defined, $H(c + x) > H(c)$ and $H(c-x) < H(c)$. All the statements follow immediately.
\end{proof}

\begin{figure}[h]
\begin{minipage}[h]{0.23\textwidth}
$a.$
\begin{center}
\vspace{0.5cm}
\includegraphics[width=0.9\textwidth]{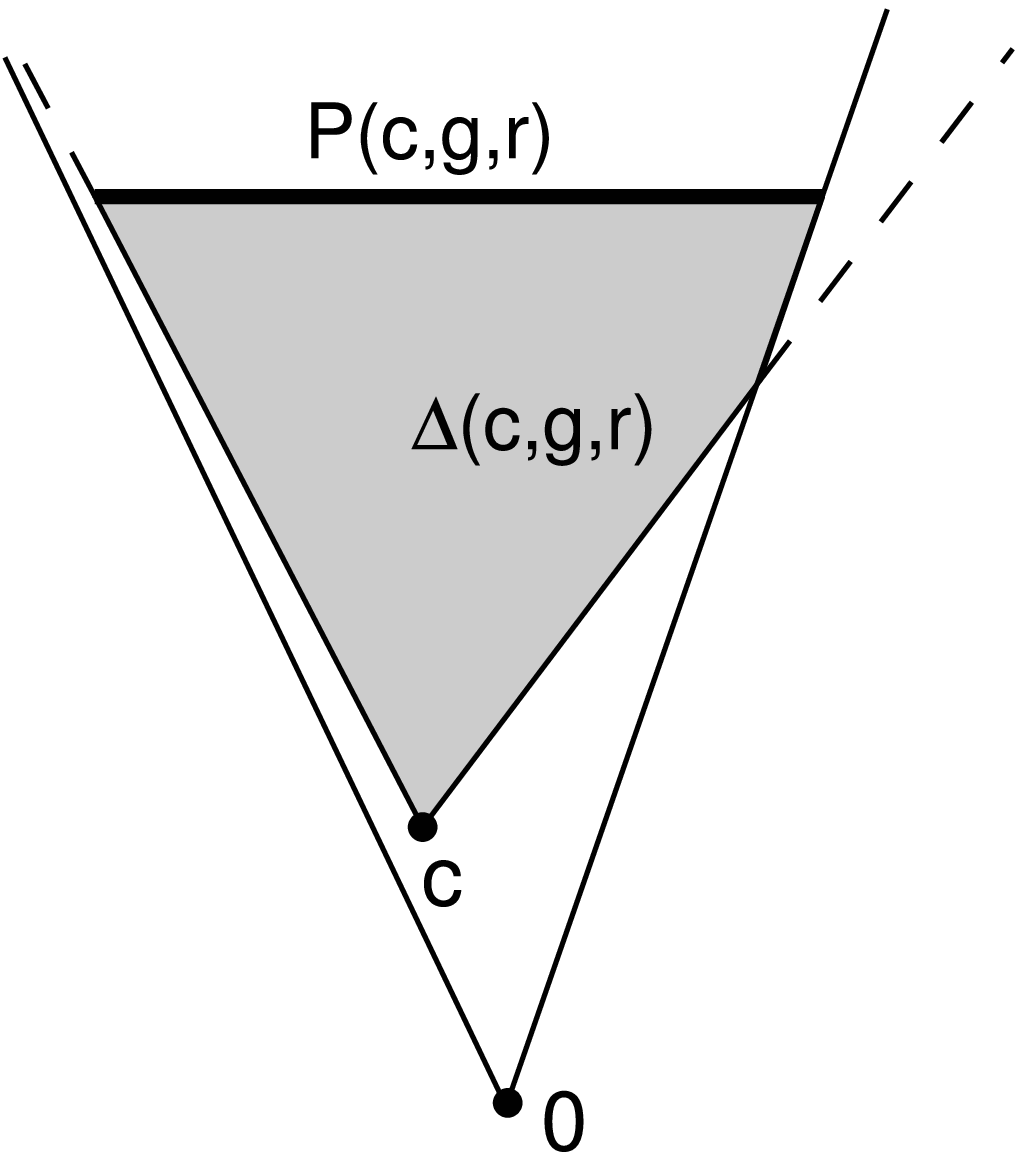}
\end{center}
\end{minipage}
\begin{minipage}[h]{0.25\textwidth}
$b.$
\begin{center}
\vspace{-0.5cm}
\includegraphics[width=0.9\textwidth]{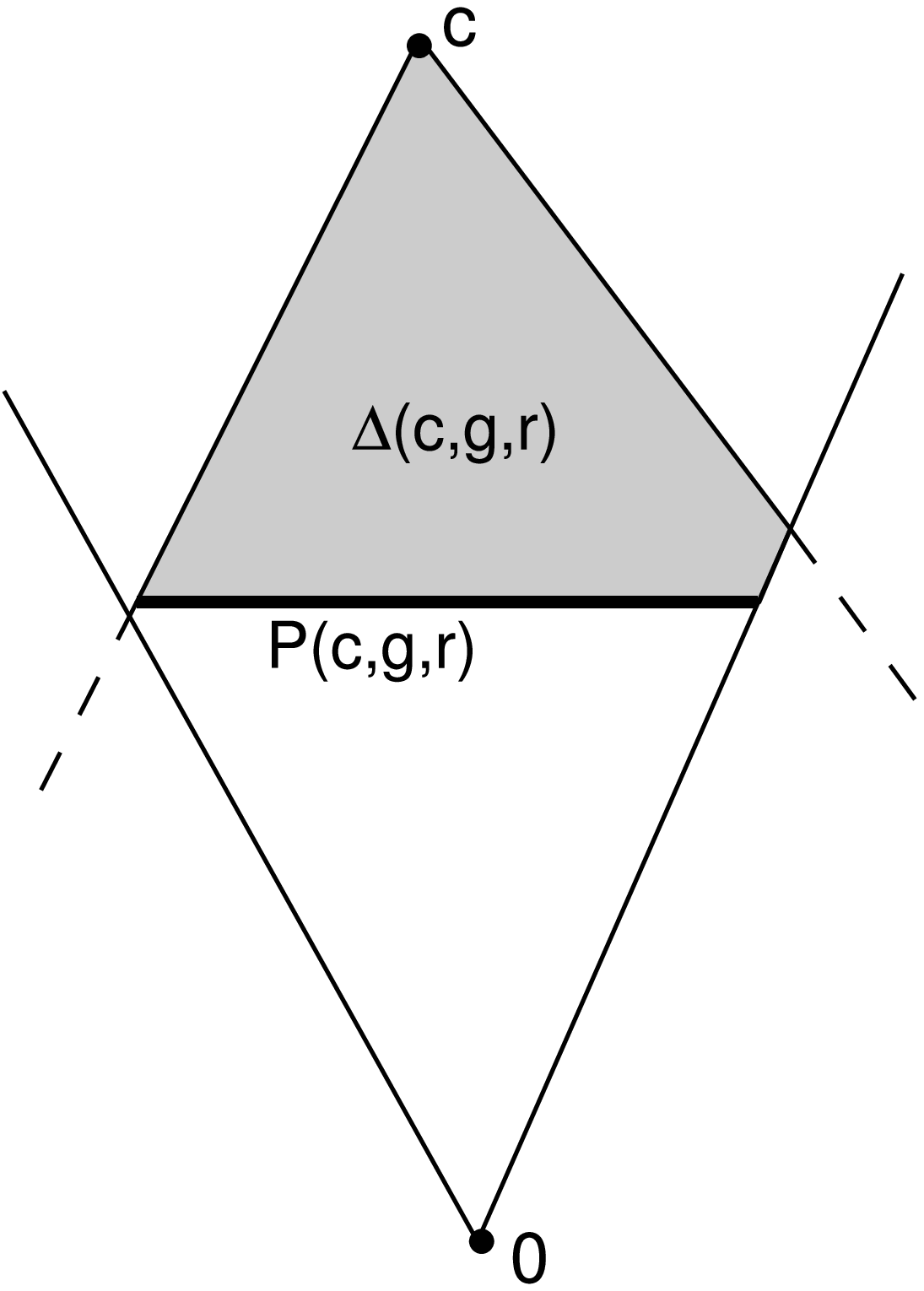}
\end{center}
\end{minipage}
\begin{minipage}[h]{0.23\textwidth}
$c.$
\begin{center}
\vspace{0.7cm}
\includegraphics[width=0.9\textwidth]{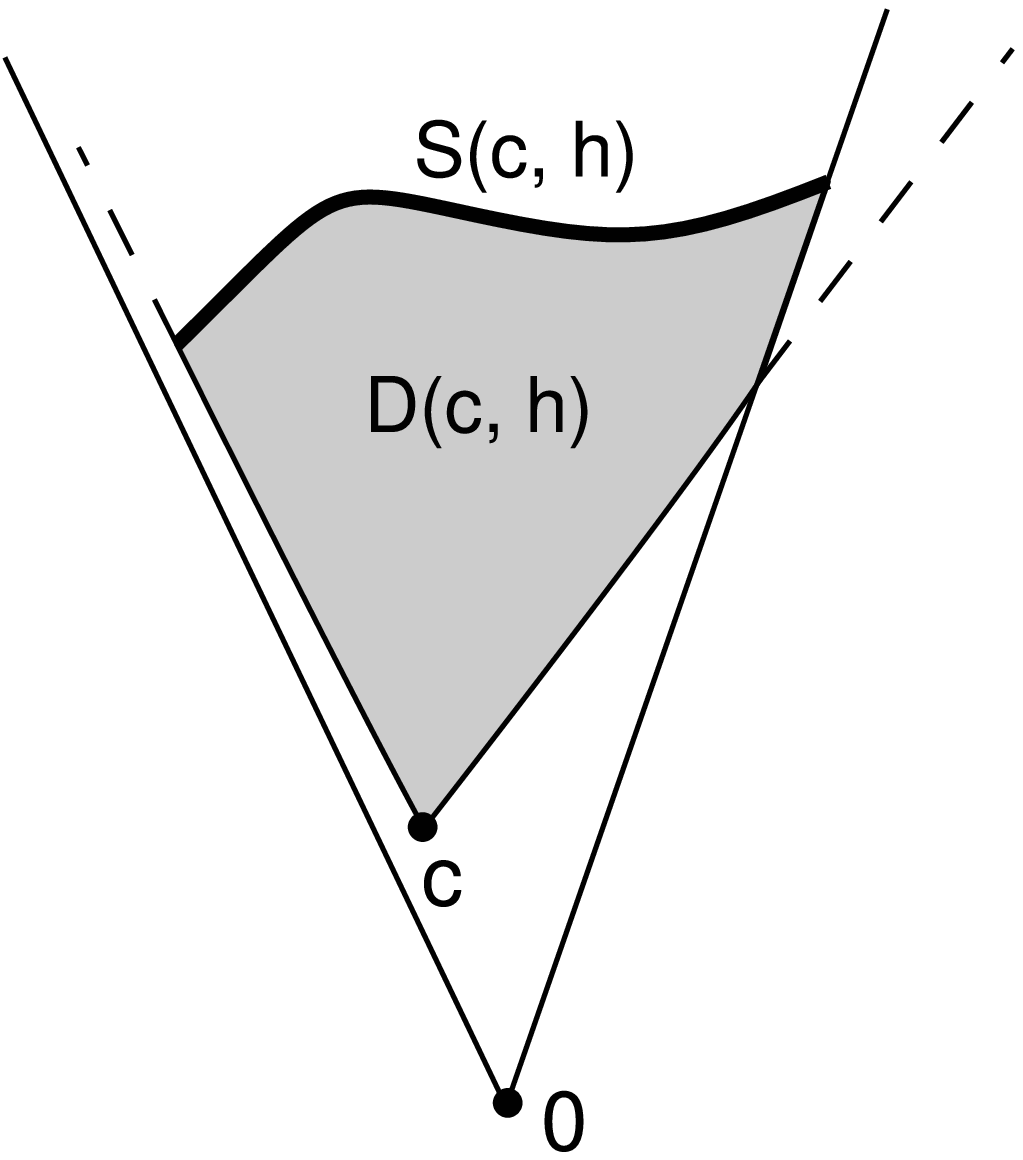}
\end{center}
\end{minipage}
\begin{minipage}[h]{0.25\textwidth}
$d.$
\begin{center}
\vspace{-0.3cm}
\includegraphics[width=0.9\textwidth]{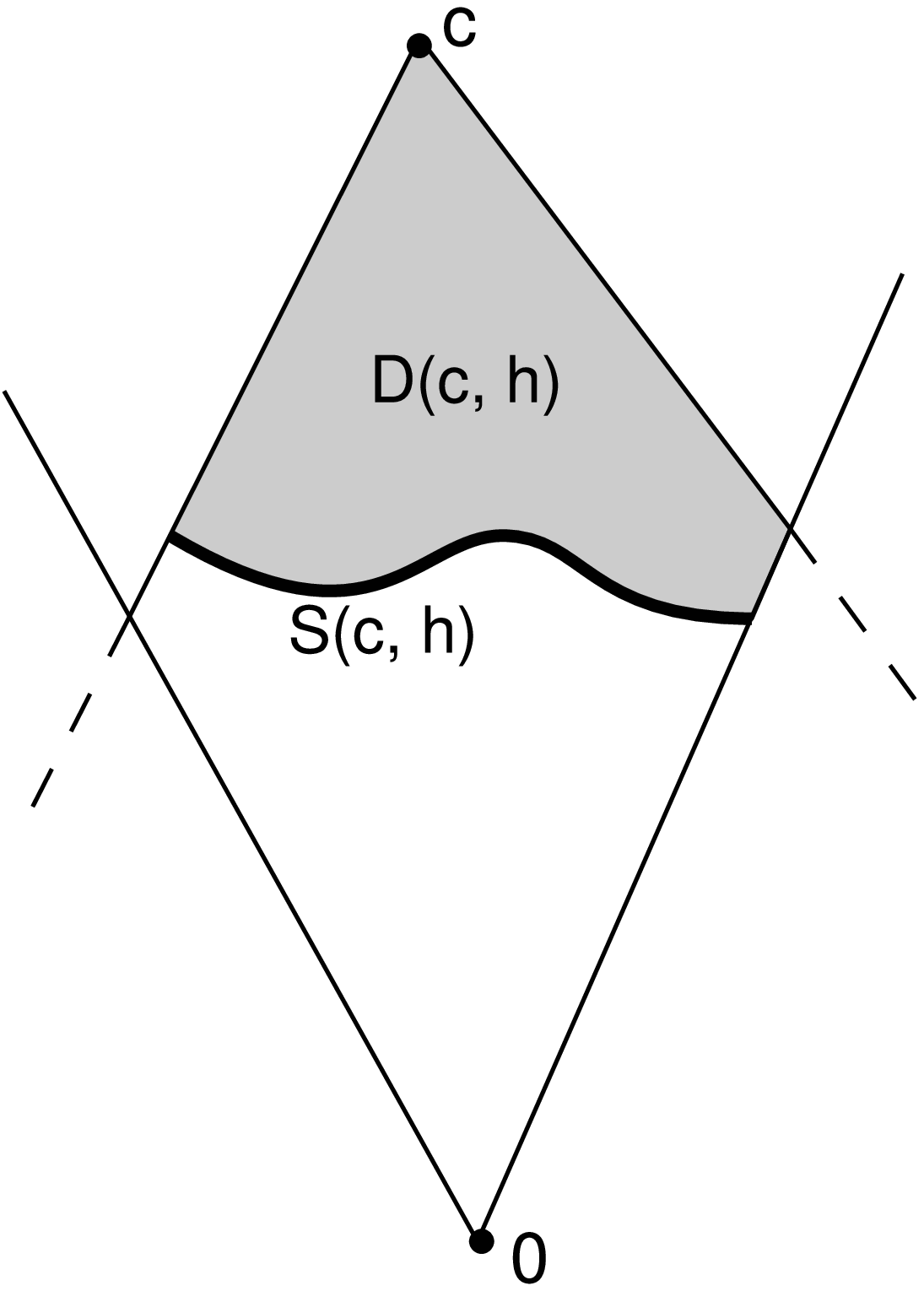}
\end{center}
\end{minipage}
\caption{\label{basic0}$a)$ and $b)$ A schematic represention of the sets $P(c, g, r)$ (bold line) and $\Delta(c, g, r)$ (shaded region). $a)$ $r > \langle g, c\rangle$, $b)$ $r < \langle g, c\rangle$. $c)$ and $d)$ A schematic represention of the sets $S(c, h)$ (bold line) and $D(c, h)$ (shaded region). $c)$ $h > H(c)$, $d)$ $h < H(c)$. }
\end{figure}

\begin{lemma}
\label{forwardinvar1}
For $c \in E$, $h \in [0, M)$, $S(c, h)$ is forward invariant.
\end{lemma}
\begin{proof}
By Lemma~\ref{forwardinvar}, $c^{+}$ and $c^{-}$ are forward invariant, and so $S_{+}(c, h)$ and $S_{-}(c,h)$ are the intersection of forward invariant sets. 
\end{proof}\\

{\bf Remark:} $S(c,h)$ may be empty for $h > H(c)$. A key milestone will be to prove that given any $c \in E$, there is some $\epsilon_c > 0$ such that for $h \in [0, H(c) + \epsilon_c)$,  $S(c, h)$ contains an equilibrium. This will follow from Lemma~\ref{forwardinvar1} after we have shown that for $h \in [0, H(c) + \epsilon_c)$, $S(c, h)$ is a nonempty ball. \\

Define $\mathbb{S}^{n-1} = \{x \in \mathbb{R}^n\,|\,|x| = 1\}$, and 
\[\delta(g) \equiv \mathrm{inf}_{y \in (Y \cap \mathbb{S}^{n-1})}\langle g, y \rangle.\] 
For any nonzero $y \in Y$, $\langle g, y \rangle > 0$. Since $Y \cap \mathbb{S}^{n-1}$ is compact, $\delta(g) > 0$. Denoting the angle between two vectors $x$ and $y$ by $\theta_{x,y}$, note that $\delta(g) = \min_{y \in Y\backslash\{0\}}\cos(\theta_{g,y})$. 


\vspace{0.25cm}
\begin{lemma}
\label{cc1}
For $r \geq 0$, $P(0, g, r)$ and $\Delta(0, g, r)$ are nonempty, compact, convex sets. For $r > 0$, $P(0, g, r)$ is an $n-1$ dimensional ball. 
\end{lemma}
\begin{proof}
Convexity and closedness of $P(0, g, r)$ and $\Delta(0, g, r)$ are immediate as each is the intersection of closed, convex sets, and the fact that the sets are nonempty follows from Lemma~\ref{notate1}. Next we prove boundedness. Suppose there is some sequence of points $(y_j) \subset P(0, g, r)$, such that $|y_j| \to \infty$. As $\langle g, y_j\rangle = r  =  |y_j| \cos(\theta_{g,y_j}) $, we must have $\cos(\theta_{g,y_j}) \to 0$ contradicting the fact that $\cos(\theta_{g,y_j}) \geq \delta(g) > 0$. This proves that $P(0, g,r)$ is bounded. Since any point $x \in \Delta(0, g, r)$ can be written $x = ty$ where $y \in P(0, g,r)$ and $0 \leq t \leq 1$, it follows that $\Delta(0, g, r)$ is bounded. 

$P(0, g, r)$ is a subset of the $n-1$ dimensional hyperplane $P(g, r)$. For $r > 0$, given any $y^{'} \in \mathrm{int}(Y)$, $y \equiv ry^{'}/\langle g, y^{'}\rangle \in P(g, r) \cap \mathrm{int}(Y)$. Consider any such $y \in P(g, r) \cap \mathrm{int}(Y)$. Take any open neighbourhood $U \subset \mathrm{int}(Y)$ of $y$ and define $U^{'} = U \cap P(g, r)$. $U^{'} \subset P(0, g, r)$ is relatively open in $P(g, r)$. So the relative interior of $P(0, g, r)$ in $P(g, r)$ is precisely $P(g, r) \cap \mathrm{int}(Y)$. Since, for $r > 0$, $P(0, g, r)$ is compact and convex with nonempty relative interior in $P(g, r)$, it is an $n-1$ dimensional ball \cite{nikaido}. 
\end{proof}\\

\begin{definition}
The {\bf diameter} of a set $X$ is $\mathrm{diam}(X) \equiv \sup_{x,y \in X}|x-y|$.
\end{definition}

\vspace{0.25cm}

In the next two lemmas we characterise the structure of the sets $P(c, g, r)$ and $\Delta(c, g, r)$ for arbitrary $c \in Y$ and $r \geq 0$. 

\begin{lemma}
\label{smalldelta}
Consider some $c \in Y$, and some $r > \langle g, c\rangle$. Then $c^{+}$ is a nonempty, closed, convex set, and $P(c,g, r)$ and $\Delta(c,g, r)$ are nonempty, compact, convex sets. Given any $\epsilon > 0$, we can choose $r> \langle g, c\rangle$ such that $\mathrm{diam}(\Delta(c, g, r)) < \epsilon$. $c^{+}$ and $\Delta(c,g, r)$ are $n$-dimensional while $P(c,g, r)$ is $n-1$ dimensional. $\mathrm{ri}(P(c,g, r)) = P(c,g, r) \cap \mathrm{int}(c^{+})$. 
\end{lemma}
\begin{proof}
Convexity and closedness of $c^{+}$, $P(c,g, r)$ and $\Delta(c,g, r)$ are immediate as they are the intersection of closed, convex sets. From Lemma~\ref{cc1}, $\Delta(0, g, r)$ is bounded, and so $P(c, g, r),\Delta(c,g, r) \subset \Delta(0, g, r)$ are bounded. Given any $y \in Y$, defining $y_2 = y(r - \langle g, c\rangle)/\langle g, y\rangle$, it is easy to check that $c+y_2 \in P(c,g,r)$, so $P(c,g, r)$, and hence $\Delta(c,g, r)$ and $c^{+}$, are nonempty. 

Fix $\epsilon > 0$ and with $\delta(g)$ defined as previously, choose $r \in (\langle g, c \rangle, \langle g, c \rangle + \epsilon \delta(g)/2)$. Consider any $y \in \Delta(c,g, r)$. When $y = c$, $|y-c| = 0 < \epsilon/2$. For $y \not = c$, rearranging $\langle g, y-c \rangle = |y-c|\cos(\theta_{g, y-c})$ gives
\[
|y-c| = \frac{\langle g, y-c \rangle}{\cos(\theta_{g, y-c})} \leq \frac{r - \langle g, c\rangle}{\delta(g)} < \epsilon/2.
\]
By compactness of $\Delta(c, g, r)$, $\sup_{y \in \Delta(c, g, r)}|y-c| < \epsilon/2$, and by the triangle inequality, $\mathrm{diam}(\Delta(c, g, r)) < \epsilon$.

Given any $y \in \mathrm{int}(Y)$, define $y_1 = ty$ for some $0 < t < \frac{r - \langle g, c\rangle}{\langle g, y\rangle}$, and $y_2 = y(r - \langle g, c\rangle)/\langle g, y\rangle$. Then, applying {\bf O2}, $c+y \in \mathrm{int}(c^{+})$, so $c^{+}$ has nonempty interior in $\mathbb{R}^n$ (i.e. it is $n$ dimensional). Any point in $\mathrm{int}(c^{+}) \cap \mathrm{int}(\Delta_{-}(g, r))$, including for example $c+y_1$, has an open neighbourhood in $\Delta(c,g, r)$, and so $\Delta(c,g, r)$ is $n$ dimensional. By arguments similar to those in Lemma~\ref{cc1}, the relative interior of $P(c, g, r)$ in $P(g, r)$ is $P(c,g, r) \cap \mathrm{int}(c^{+})$, which contains the point $c + y_2$, and so is nonempty. So $P(c, g, r)$ is $n-1$ dimensional, $\mathrm{Aff}(P(c, g, r)) = P(g, r)$, and hence $\mathrm{ri}(P(c,g, r)) = P(c,g, r) \cap \mathrm{int}(c^{+})$. 
\end{proof}\\

Note that $\partial c^{+} = (c^{+} \cap (c+\partial K)) \cup (c^{+} \cap \partial Y)$. The analogous lemma for $r < \langle g, c \rangle$ is slightly altered by the fact that $c^{-}$ may have empty interior in $\mathbb{R}^n$:
\begin{lemma}
\label{smalldelta1}
Consider some $c \in Y\backslash\{0\}$, and some $r \in [0,  \langle g, c\rangle)$. Then $c^{-}$, $P(c,g, r)$ and $\Delta(c,g, r)$ are nonempty, compact, convex sets. Given any $\epsilon > 0$, we can choose $r \in [0, \langle g, c\rangle)$ such that $\mathrm{diam}(\Delta(c, g, r)) < \epsilon$. $\Delta(c,g, r)$ has nonempty relative interior in $\mathrm{Aff}(c^{-})$. When $r > 0$, $\mathrm{ri}(P(c,g, r)) = P(g, r) \cap \mathrm{ri}(c^{-})$. 
\end{lemma}
\begin{proof}
$c^{-}$ is closed and convex by construction. As $[0, c] \subset c^{-}$, it is nonempty (and at least 1 dimensional). Since $c^{-} \subset \Delta(0, g, \langle g, c\rangle)$, by Lemma~\ref{cc1}, it is bounded. $P(c,g, r)$ and $\Delta(c,g, r)$ are convex and closed as the intersection of such sets, and are bounded as subsets of $c^{-}$. Since $rc/\langle g, c \rangle \in P(c,g, r)$, so $P(c,g, r)$, and hence $\Delta(c,g,r)$, are nonempty. With $\delta(g)$ defined as previously, choose any $r > 0$ satisfying $r \in (\langle g, c \rangle - \epsilon \delta(g)/2, \langle g, c \rangle)$. Consider any vector $y \in \Delta(c,g, r)$. When $y = c$, $|y-c| = 0 < \epsilon$. For $y \not = c$, rearranging $\langle g, c-y \rangle = |c-y|\cos(\theta_{g, c-y})$ gives
\[
|c-y| = \frac{\langle g, c-y \rangle}{\cos(\theta_{g, c-y})} \leq \frac{\langle g, c\rangle - r}{\delta(g)} < \epsilon/2\,.
\]
By compactness of $\Delta(c, g, r)$, $\max_{y \in \Delta(c, g, r)}|c-y| < \epsilon/2$, and by the triangle inequality $\mathrm{diam}(\Delta(c, g, r)) < \epsilon$.

As $c^{-}$ is convex and contains both $c$ and $0$, and $\langle g, \cdot \rangle$ is continuous, $\langle g, \cdot \rangle$ takes all values in $(0, \langle g, c \rangle)$ in $\mathrm{ri}(c^{-})$. 
Consider any $y_1 \in \mathrm{ri}(c^{-})$ such that $\langle g, y_1 \rangle \in (r, \langle g, c \rangle)$. Take an open neighbourhood $U\subset \mathrm{int}(\Delta_{+}(g, r))$ of $y_1$, such that $U^{'} = U \cap \mathrm{Aff}(c^{-}) \subset c^{-}$. 
Then $U^{'} \subset \Delta(c,g, r)$, showing that $\Delta(c,g, r)$ has nonempty relative interior in $\mathrm{Aff}(c^{-})$. Thus $\mathrm{Aff}(c^{-}) = \mathrm{Aff}(\Delta(c,g, r))$, and $\mathrm{ri}(\Delta(c,g, r)) = \mathrm{ri}(c^{-}) \cap \mathrm{int}(\Delta_{+}(g, r))$. 

Fix $r > 0$ and choose any $y_2 \in \mathrm{ri}(c^{-})$ such that $\langle g, y_2 \rangle  = r$. Take any open neighbourhood $U$ of $y_2$ such that $U \cap \mathrm{Aff}(c^{-}) \subset c^{-}$. 
Then $U^{'} = U \cap \mathrm{Aff}(c^{-}) \cap P(g, r) \subset P(c, g, r)$. Thus $y_2 \in \mathrm{ri}(P(c, g, r))$, and in fact $\mathrm{ri}(P(c,g, r)) = P(g, r) \cap \mathrm{ri}(c^{-})$. 
\end{proof}\\

{\bf Remarks.} The fact that $c^{-}$ may have empty interior in $\mathbb{R}^n$ necessitates some care in the arguments. However, once attention is restricted to $\mathrm{Aff}(c^{-})$, the fundamental geometrical notions are similar to the case of $c^{+}$: define $Y_c \equiv Y \cap \mathrm{Aff}(c^{-})$ and $K_c \equiv K \cap \mathrm{Aff}(c^{-})$. Note that $c^{-} = (c-K_c) \cap Y_c$, and since $Y_c \cap c^{-} = c^{-} = K_c \cap c^{-}$, so both $Y_c, K_c \supset c^{-}$, and thus both have nonempty relative interior in $\mathrm{Aff}(c^{-})$. Further, $\mathrm{relbd}(c^{-})$ is the union of $c^{-} \cap (c - \mathrm{relbd}(K_c))$ and $c^{-} \cap \mathrm{relbd}(Y_c)$.
The fact that for $r \in (0, \langle g, c \rangle)$, $\mathrm{ri}(P(c,g, r)) = P(c,g, r) \cap \mathrm{ri}(c^{-})$ and for $r > \langle g, c \rangle$, $\mathrm{ri}(P(c,g, r)) = P(c,g, r) \cap \mathrm{int}(c^{+})$, motivates the definition:
\[
T(c,g,r) \equiv \left\{ \begin{array}{ll}g^\perp \cap \mathrm{Aff}(c^{-}), & r < \langle g, c \rangle\\
g^\perp, & r > \langle g, c \rangle\,.\end{array} \right .
\]
Geometrically, $T(c,g,r)$ is the tangent space to $P(c, g, r)$ provided $r \not  \in \{0, \langle g, c\rangle\}$, that is, given $c \not = 0$, $r \in (0, \langle g, c \rangle) \cup (\langle g, c \rangle, \infty)$, $x \in P(c, g, r)$, and some $\delta \in \mathbb{R}^n$, then $x + \delta \in \mathrm{Aff}(P(c, g, r))$ iff $\delta \in T(c,g,r)$. If $x \in \mathrm{ri}(P(c, g, r))$, then there exists $t > 0$ such that $x + t\delta \in P(c, g, r)$ iff $\delta \in g^\perp \cap \mathrm{Aff}(c^{-})$. \\

We now characterise $S(c, h)$ and $D(c, h)$.
\begin{lemma}
\label{boundedD}
Given any $c \in Y$, and any $\epsilon > 0$, there is some $h^{'} > H(c)$ such that for all $h \in (H(c), h^{'})$, $S(c,h)$ and $D(c,h)$ are nonempty and compact with $\mathrm{diam}(D(c,h)) < \epsilon$.  
\end{lemma}
\begin{proof}
$S(c,h)$, $D(c,h)$ are closed by construction, and $S(c,h)$ is bounded provided $D(c,h)$ is bounded. By  Lemma~\ref{smalldelta}, choose $r > \langle g, c \rangle$ such that $\mathrm{diam}(\Delta(c, g, r)) < \epsilon$. For $x \in P(c, g, r)$, $x > c$, and hence $H(x) > H(c)$. By continuity of $H$ and compactness of $P(c, g, r)$, we get that $h^{'} \equiv \inf\{H(x)\,|\,x \in P(c, g, r)\} > H(c)$. Choosing $h \in (H(c), h^{'})$ and any $x \in P(c, g, r)$, and applying the intermediate value theorem along the line segment $[c, x]$, we see that there exists $x_h \in [c, x]$ such that $H(x_h) = h$, and so $S(c, h)$ is nonempty. Suppose there exists $x \in S(c, h) \backslash \Delta(c,g,r)$. Since $[c, x] \subset c^{+}$, in order to exit from $\Delta(c,g,r)$, $[c, x]$ must intersect $P(c,g,r)$ at some point $x^{'} < x$. Since $h^{'} > h$ we must have $h = H(x) < h^{'} \leq H(x^{'})$, contradicting the fact that $x^{'} < x$ implies $H(x^{'}) < H(x)$. So $S(c, h) \subset \Delta(c,g,r)$. The same argument applies for any $\tilde h  \in (H(c), h]$, and since $D(c, h) = \cup_{\tilde h \leq h}S(c, h)$,  $D(c, h) \subset \Delta(c,g,r)$ and $\mathrm{diam}(D(c,h)) < \epsilon$. 
\end{proof}\\

\begin{lemma}
\label{boundedD1}
Given any $c\in Y$, for all $h \in [0, H(c))$, $S(c,h)$ and $D(c,h)$ are nonempty and compact. If $c \not = 0$, given any $\epsilon > 0$, we can choose $h^{'} \in [0, H(c))$ such that for all $h \in (h^{'}, H(c))$, $\mathrm{diam}(D(c,h)) < \epsilon$. 
\end{lemma}
\begin{proof}
$S(c,h)$ and $D(c,h)$ are closed by construction. By Lemma~\ref{smalldelta1}, $c^{-}$ is compact, and so $S(c,h)$ and $D(c,h)$ are bounded. Applying the intermediate value theorem along the line segment $[0, c]$, we see that for any $h \in [0, H(c)]$, there exists $x_h \in [0, c]$ such that $H(x_h) = h$, and so $S(c, h)$ is nonempty. By  Lemma~\ref{smalldelta1}, choose $r \in [0, \langle g, c \rangle)$ such that $\mathrm{diam}(\Delta(c, g, r)) < \epsilon$. Any point $x \in P(c, g, r)$ satisfies $x < c$, and hence $H(x) < H(c)$. By continuity of $H$ and compactness of $P(c, g, r)$, we get that $h^{'} \equiv \sup\{H(x)\,|\,x \in P(c, g, r)\} < H(c)$. Choose $h \in (h^{'}, H(c))$. Suppose there exists $x \in S(c, h) \backslash \Delta(c,g,r)$. Since $[c,x] \subset c^{-}$, in order to exit from $\Delta(c,g,r)$, $[c,x]$ must intersect $P(c,g,r)$ at some point $x^{'} > x$. Since $h^{'} < h$ we must have $h = H(x) > h^{'} \geq H(x^{'})$, contradicting the fact that $x^{'} > x$ implies $H(x^{'}) > H(x)$. So $S(c, h) \subset \Delta(c,g,r)$. The same argument applies for any $\tilde h  \in [h, H(c))$, and since $D(c, h) = \cup_{\tilde h \leq h}S(c, h)$,  $D(c, h) \subset \Delta(c,g,r)$ and $\mathrm{diam}(D(c,h)) < \epsilon$. 
\end{proof}\\

\begin{lemma}
\label{boundedP}
Given any $c \in Y$, $h > H(c)$, there is some $r > \langle g, c\rangle$ such that $\Delta(c, g,r)$ lies in $D(c,h)$, and $\max_{x \in \Delta(c, g, r)}H(x) < h$.
\end{lemma}
\begin{proof}
By continuity of $H$ at $c$ there is some $\epsilon>0$ such that $|x-c| < \epsilon$ implies that $|H(x) - H(c)| < h - H(c)$. By Lemma~\ref{smalldelta} we can choose $r>\langle g, c \rangle$ such that $\Delta(c,g,r)$ has nonempty interior and $\mathrm{diam}(\Delta(c,g,r)) < \epsilon$, i.e. $H(x) < h$ for $x \in \Delta(c,g,r)$. Thus $\Delta(c, g, r) \subset D(c,h)$, and by compactness of $\Delta(c, g, r)$, $\mathrm{max}_{x \in \Delta(c, g, r)}H(x) < h$. 
\end{proof}\\

\begin{lemma}
\label{boundedP1}
Given any $c \in Y\backslash\{0\}$, $h \in (0, H(c))$, there is some $r \in [0, \langle g, c\rangle)$ such that $\Delta(c, g,r)$ lies in $D(c,h)$, and $\mathrm{max}_{x \in \Delta(c, g, r)}H(x) > h$.
\end{lemma}
\begin{proof}
By continuity of $H$ at $c$ there is some $\epsilon$ such that $|x-c| < \epsilon$ implies that $|H(x) - H(c)| < H(c) - h$. By Lemma~\ref{smalldelta1} we can choose $r \in [0, \langle g, c\rangle)$ such that $\mathrm{diam}(\Delta(c,g,r)) < \epsilon$, i.e. $H(x) > h$ for $x \in \Delta(c,g,r)$. Thus $\Delta(c, g, r) \subset D(c,h)$, and by compactness of $\Delta(c, g, r)$, $\mathrm{max}_{x \in \Delta(c, g, r)}H(x) > h$. 
\end{proof}\\

Lemmas~\ref{boundedD}~and~\ref{boundedP} will be used as follows: 
\begin{enumerate}
\item Given any $c\in Y$, $k_2 > \langle g, c \rangle$, we construct the bounded convex set $\Delta(c,g,k_2)$.
\item We then choose $h \in (H(c), \inf\{H(x)\,|\,x \in P(c, g, k_2)\})$ so that $D(c,h) \subset \Delta(c,g,k_2)$ (Lemma~\ref{boundedD}).
\item Thirdly we choose $k_1$ satisfying $\langle g, c \rangle < k_1 < k_2$ such that $\Delta(c, g, k_1) \subset D(c,h)$ (Lemma~\ref{boundedP}). Thus $P(c, g, k_2)$ and $P(c, g, k_1)$ ``trap'' $S(c, h)$. 
\end{enumerate}
The construction is illustrated in Figure~\ref{stages}. A ray originating in $\Delta(c, g, k_1)$ and intersecting $P(c, g, k_2)$ must first intersect both $P(c, g, k_1)$ and $S(c, h)$. Similarly any ray originating in $\Delta(c, g, k_1)$ and intersecting $S(c, h)$ must first intersect $P(c, g, k_1)$. By results to follow, this last fact will imply that $S(c, h)$ is homeomorphic to a subset of $P(c, g, k_1)$, which can be shown to be a ball. An analogous construction follows from Lemmas~\ref{boundedD1} and \ref{boundedP1}. 

\begin{figure}[h]
\begin{minipage}[h]{0.32\textwidth}
\begin{center}
\includegraphics[width=0.9\textwidth]{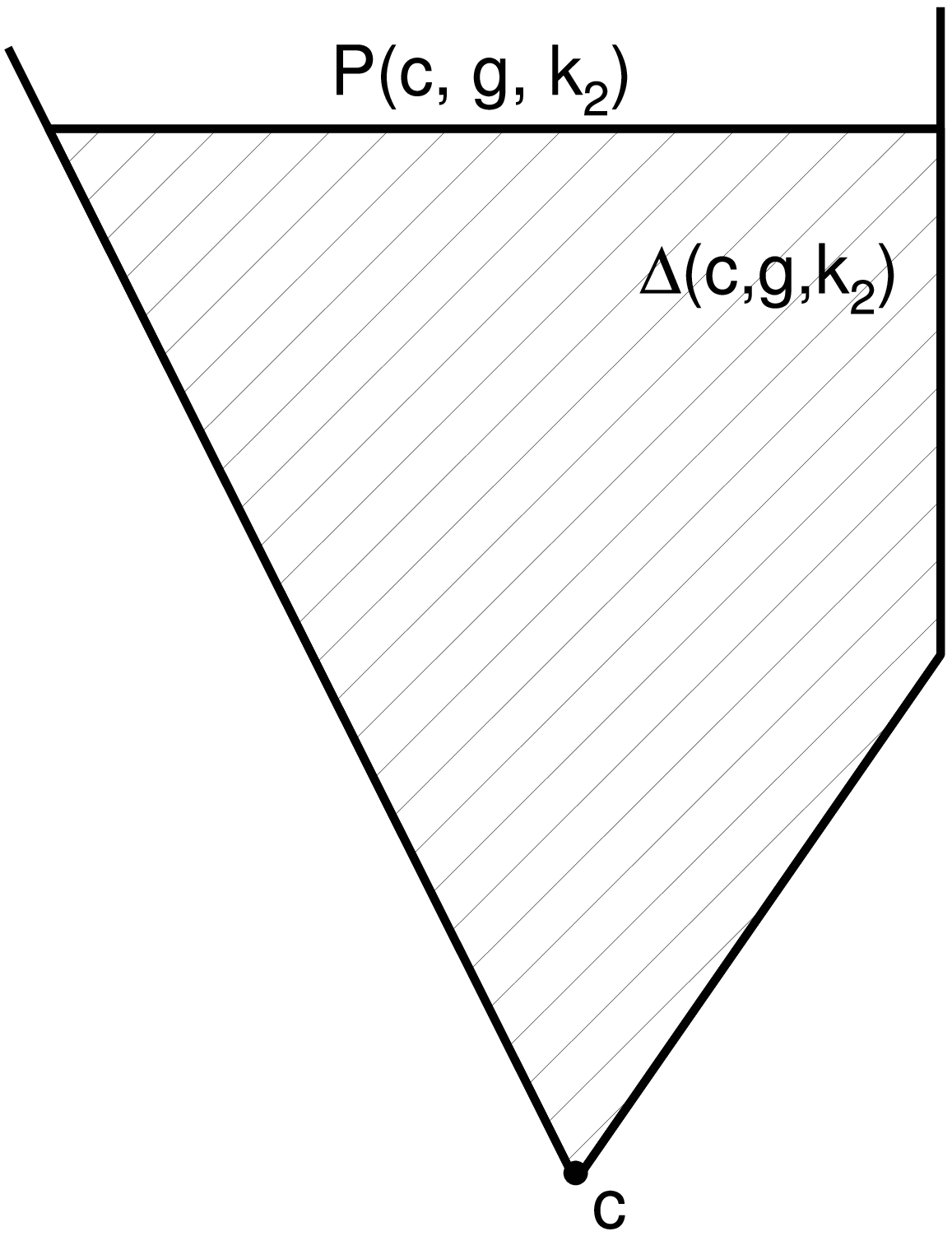}
\end{center}
\end{minipage}
\hfill
\begin{minipage}[h]{0.32\textwidth}
\begin{center}
\includegraphics[width=0.9\textwidth]{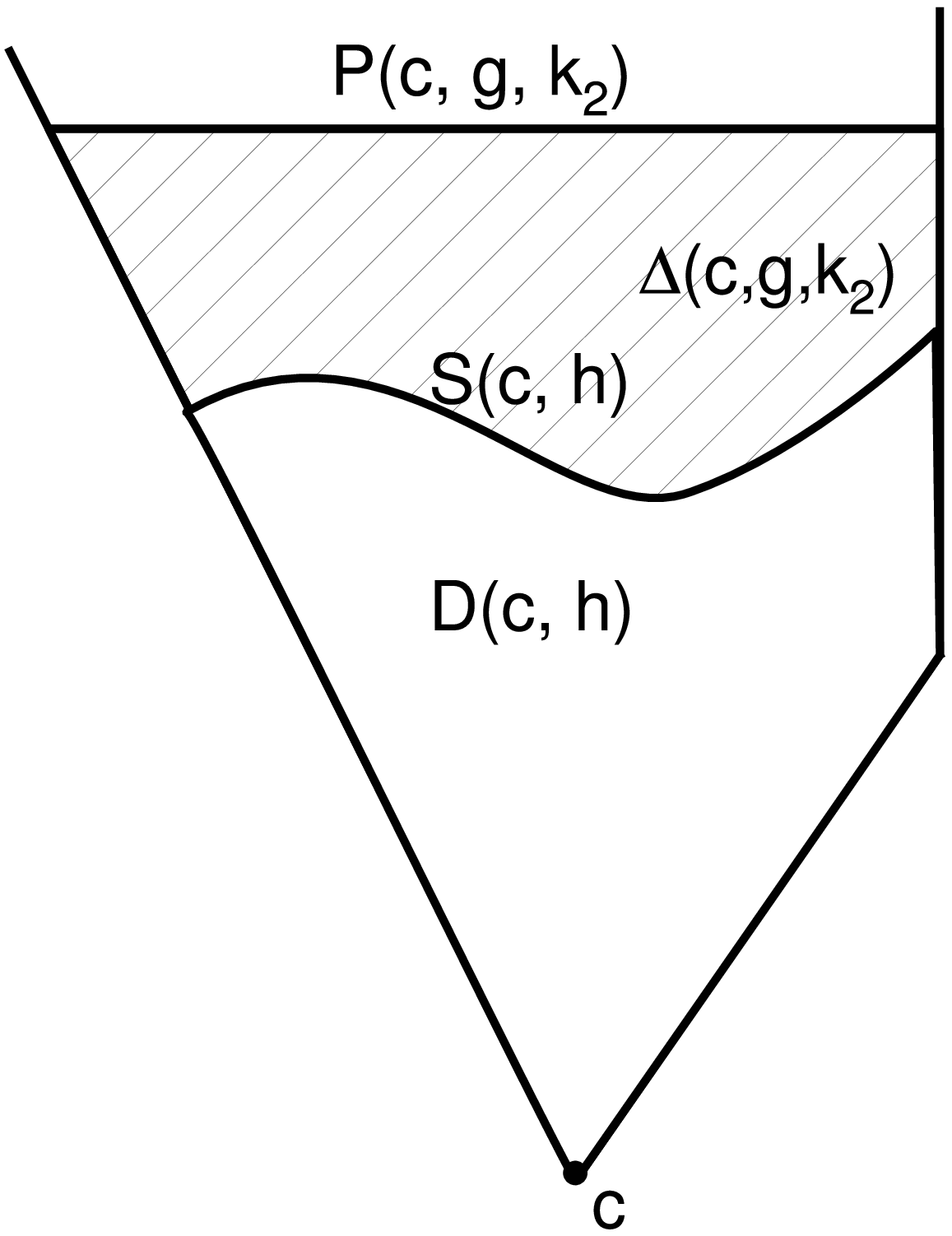}
\end{center}
\end{minipage}
\hfill
\begin{minipage}[h]{0.32\textwidth}
\begin{center}
\includegraphics[width=0.9\textwidth]{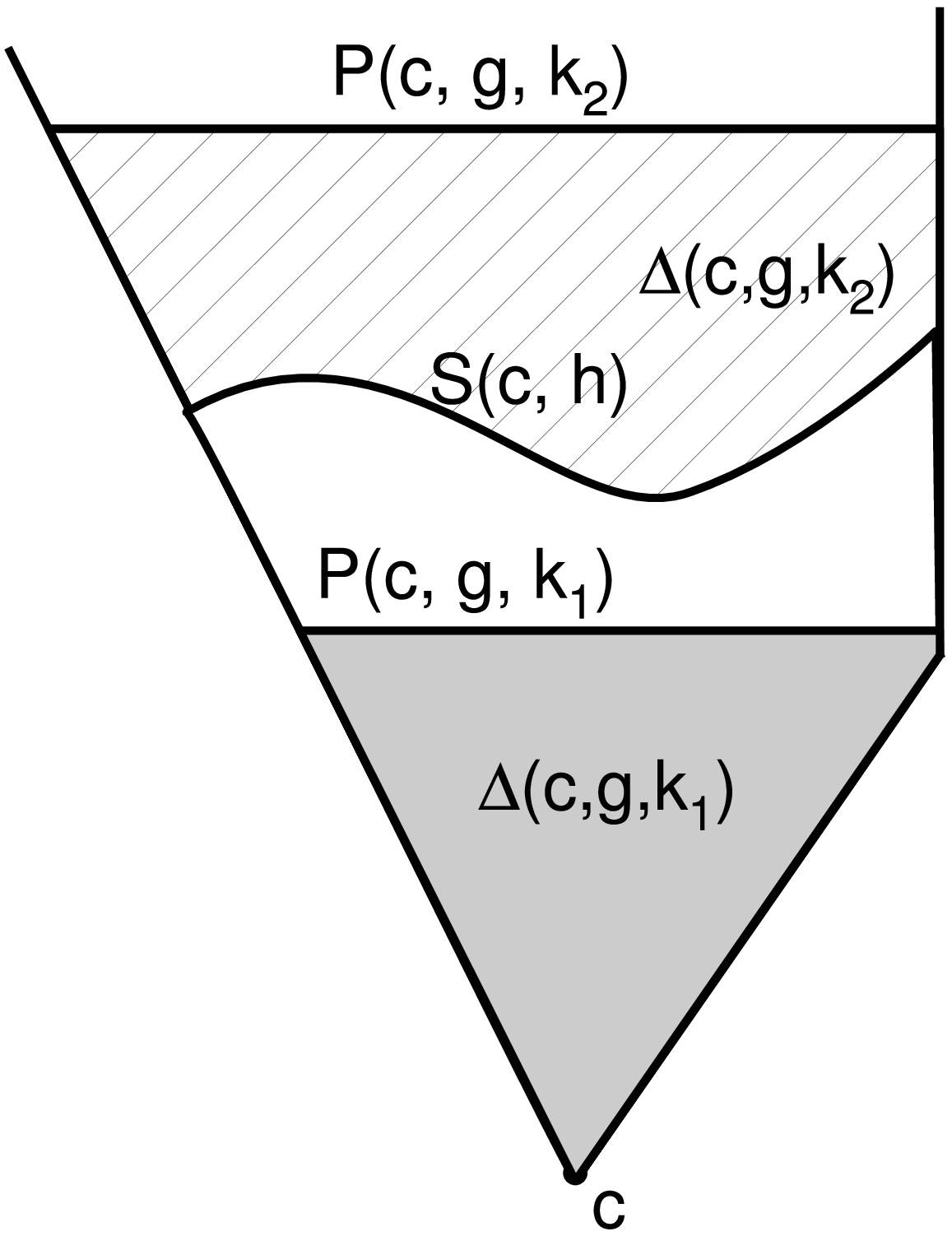}
\end{center}
\end{minipage}
\hfill
\begin{minipage}[h]{\textwidth}
\caption{\label{stages}A schematic representation of the construction of sets $P(c, g, k_1)$ and $P(c, g, k_2)$ trapping a set $S(c, h)$. {\em Left}. Given arbitrary $c, g$ and $k_2 > \langle g, c \rangle$, $\Delta(c,g,k_2)$ is constructed (hatched region). {\em Middle.} $h$ is chosen so that $S(c,h)$, and hence $D(c,h)$ (white region), lies inside $\Delta(c,g,k_2)$. {\em Right.} $k_1$ is chosen so that $S(c,h)$ is trapped between $P(c, g, k_1)$ and $P(c, g, k_2)$ and hence $\Delta(c, g, k_1)$ (shaded region) lies inside $D(c, h)$.}
\end{minipage}
\end{figure}

\section{Central projections}

We are working towards proving that for any $c \in E$, there is some $\epsilon_c > 0$ such that for any $h \in [0, H(c) + \epsilon_c)$, $S(c, h)$ is a ball. As in \cite{mierczynski}, a homeomorphism between $S(c, h)$ and a compact, convex set will be constructed via projections, for which we need some basic ideas developed in this section.

Any $c \in \mathbb{R}^n$, $X \subset \mathbb{R}^n$ define a natural cone $K(c, X)= \cup_{x \in X}r_{c \to x}$. $K(c, X)$ is not necessarily closed, convex, pointed or solid.

\vspace{0.25cm}
\begin{definition}
Given a point $c$, a set $X$ disjoint from $c$ and such that each ray originating at $c$ intersects $X$ at most once, define the projection $\Pi_{c, X}:K(c, X)\backslash\{c\}\mapsto X$ by $\Pi_{c, X}(y) \equiv r_{c \to y} \cap X$. 
\end{definition}
\vspace{0.25cm}

All discussion in this section is translation invariant, and it is convenient to assume, without loss of generality, that $c = 0$, and write $K(X)$ for $K(0, X)$. It is also useful to define: 
\[
K_{+}(X) \equiv \{tx\,|\, t \in (1,\infty), x \in X\}, \quad K_{-}(X) \equiv \{tx\,|\, t \in (0, 1), x \in X\}\,.
\]

\vspace{0.25cm}
\begin{lemma}
\label{basicproj}
Consider a set $X \subset \mathbb{R}^n\backslash\{0\}$, such that $K_{+}(X)$ and $K_{-}(X)$ are relatively open in $K(X)$. Assume that for each $x \in X$, $r_{0\to x} \cap X = \{x\}$ (i.e. the ray $r_{0\to x}$ intersects $X$ exactly once).  The projection $\Pi_{0, X}:K(X)\backslash\{0\}\mapsto X$ is continuous. 
\end{lemma}
\begin{proof}
Given any $x \in K(X)\backslash\{0\}$ define $t(x)$ via $\Pi_{0, X}(x) = t(x)x$. Fix $x \in K(X)\backslash\{0\}$ and any $\epsilon \in (0, t(x))$. Let $t_1 = t(x) - \epsilon$, and $t_2 = t(x) + \epsilon$. By construction $t_1x \in K_{-}(X)$ and $t_2x \in K_{+}(X)$. Let $\epsilon_1$ and $\epsilon_2$ be the diameters of relatively open neighbourhoods of $t_1x$ in $K_{-}(X)$ and $t_2x$ in $K_{+}(X)$ respectively. Define $\delta = \min\{\epsilon_1/t_1, \epsilon_2/t_2\}$, and choose any $y \in K(X)\backslash\{0\}$ such that $|y-x| < \delta$. We get that $|t_1y - t_1x| < \epsilon_1$, so that $t_1y \in K_{-}(X)$. Similarly $|t_2y - t_2x| < \epsilon_2$, i.e. $t_2y \in K_{+}(X)$. Thus $t_1 < t(y) < t_2$, i.e. $|t(y) - t(x)| < \epsilon$. Thus $t(x)$ is continuous, and hence $\Pi_{0, X}$ is continuous. 
\end{proof}\\

\begin{lemma}
\label{compactX}
Consider a compact set $X \subset \mathbb{R}^n\backslash\{0\}$ such that for each $x \in X$, $r_{0\to x} \cap X = \{x\}$. Then $\Pi_{0, X}:K(X)\backslash\{0\}\mapsto X$ is continuous. 
\end{lemma}
\begin{proof}
Continuity of $\Pi_{0, X}$ follows immediately from Lemma~\ref{basicproj} provided $K_{+}(X)$ and $K_{-}(X)$ are relatively open in $K(X)$. Assume first that $K_{-}(X)$ is not relatively open. This means that there is a point $q \in K_{-}(X)$ such that every neighbourhood of $q$ contains points in $K_{+}(X)$. 
Since $q \in K_{-}(X)$, for some fixed $t > 1$, $tq \in X$. Take a sequence of points $q_i \to q$ with $q_i \in K_{+}(X)$, and the sequence of values $\Pi_{0, X}(q_i)=t_iq_i \in X$ with $t_i < 1$. $(t_i)$ is a bounded real sequence, and by passing to a subsequence if necessary, we get a convergent sequence of values $t_{i_k}$ such that $\lim_{k \to \infty} t_{i_k} = t^{'} \leq 1$. Thus $t_{i_k}q_{i_k} \to t^{'}q \in \mathrm{cl}(X)$. As each ray intersects $X$ exactly once and $tq \in X$, $t^{'}q \not \in X$ and $X$ is not closed. 

Now assume that $K_{+}(X)$ is not relatively open in $K(X)$, i.e. there is some  point $q \in K_{+}(X)$ and a sequence of points $q_i \to q$ with $q_i \in K_{-}(X)$. Since $q \in K_{+}(X)$, for some fixed $t < 1$, $tq \in X$. Define the sequence of values $\Pi_{0, X}(q_i)=t_iq_i \in X$ with $t_i > 1$. Since $q_i$ is bounded away from zero 
and $X$ is bounded, $(t_i)$ is bounded, and by passing to a subsequence if necessary, we get a convergent sequence of values $t_{i_k}$ such that $\lim_{k \to \infty} t_{i_k} = t^{'} \geq 1$. Thus $t_{i_k}q_{i_k} \to t^{'}q \in \mathrm{cl}(X)$. As each ray intersects $X$ exactly once and $tq \in X$, again $X$ is not closed. 
\end{proof}\\

\begin{lemma}
\label{convexproj}
Consider a compact, convex set $X \subset \mathbb{R}^n$, with $0 \in \mathrm{ri}(X)$. The projection $\Pi_{0, \partial X}:K(\partial X)\backslash\{0\}\mapsto \partial X$ is well defined and continuous. 
\end{lemma}
\begin{proof}
By {\bf O2}, for any $x \in \mathrm{Aff}(X)$, $r_{0 \to x}$ intersects $\partial X$ at exactly one point, and so $K(\partial X) = \mathrm{Aff}(X)$, and $\Pi_{0, \partial X}$ is well defined. $K_{+}(\partial X) = \mathrm{Aff}(X)\backslash \mathrm{cl}(X)$ and $K_{-}(\partial X) = \mathrm{ri}(X)\backslash\{0\}$ are open, and so by Lemma~\ref{basicproj}, $\Pi_{0, \partial X}$ is continuous. 
\end{proof}\\

\begin{lemma}
\label{twosets}
Consider a compact set $X \subset \mathbb{R}^n\backslash\{0\}$ and some bounded set $Y \subset \mathbb{R}^n\backslash\{0\}$. Assume that for each $x \in X$, $r_{0\to x} \cap X = \{x\}$ and $r_{0\to x} \cap Y$ is a singleton. Then the set $Y_0 \equiv \Pi_{0, Y}(X)$ is homeomorphic to $X$. 
\end{lemma}
\begin{proof}
By construction, $\left.\Pi_{0, Y}\right|_X$ is a bijection between $X$ and $Y_0$. By Lemma~\ref{compactX}, the inverse mapping $\left.\Pi_{0, X}\right|_{Y_0}$  is continuous on $Y_0$, and so $Y_0$ is closed. Since $Y$ is bounded, $Y_0$ is compact. Applying Lemma~\ref{compactX} again, $\left.\Pi_{0, Y}\right|_X$ is continuous, and so $Y_0$ and $X$ are homeomorphic. 
\end{proof}\\

\begin{lemma}
\label{lemball}
Consider a nonempty, compact set $X\subset \mathbb{R}^n$ with $0 \in \mathrm{ri}(X)$, and such that for each $x \in \mathrm{relbd}(X)$, $r_{0 \to x} \cap \mathrm{relbd}(X) = \{x\}$. Then $X$ is a ball. 
\end{lemma}

\begin{proof}
If $\mathrm{Aff}(X)$ has dimension $0$, $X$ is a single point, which is by our definition a $0$-dimensional ball. Assume $\mathrm{Aff}(X)$ has dimension $k > 0$. By compactness of $X$, each nontrivial ray originating at $0$ must eventually enter $\mathrm{Aff}(X)\backslash X$, and hence must intersect $\mathrm{relbd}(X)$: so $K(\mathrm{relbd}(X)) = \mathrm{Aff}(X)$. As $\mathrm{relbd}(X)$ is compact, we can apply Lemma~\ref{compactX} to $\mathrm{relbd}(X)$, giving that $\Pi_{0, \mathrm{relbd}(X)}: \mathrm{Aff}(X)\backslash\{0\} \to \mathrm{relbd}(X)$ is continuous. It follows that $l(x) \equiv |\Pi_{0, \mathrm{relbd}(X)}(x)|$ is continuous on $\mathrm{Aff}(X)\backslash\{0\}$, and by compactness of $ \mathrm{relbd}(X)$,
\[
0 < \min_{x \in \mathrm{relbd}(X)}|x| \equiv l_{min} \leq l(x) \leq l_{max} \equiv \max_{x \in \mathrm{relbd}(X)}|x|< \infty\,.
\]
Let $B = \{x \in \mathrm{Aff}(X)\,|\,\,|x| \leq 1\}$. Clearly $G:\mathrm{Aff}(X)\backslash\{0\} \to B$ defined by $G(x) =  x/l(x)$ takes $X\backslash\{0\}$ homeomorphically to $B\backslash\{0\}$. Defining $G(0) = 0$, consider any sequence of points $x_i \to 0$. Then $G(x_i) = x_i/l(x_i) \leq x_i/l_{min} \to 0$, and $G^{-1}(x_i) = l(x_i)x_i \leq l_{max}x_i  \to 0$. So $G$ is a homeomorphism between $X$ and $B$, i.e. $X$ is a $k$-dimensional ball. 
\end{proof}\\

\section{Main results}

The first, easy use of the central projections discussed in the previous section is that bounded, nonempty level sets of $H$ contain equilibria:

\vspace{0.25cm}
\begin{lemma}
\label{boundedeq}
For $h \in [0, M)$, if $S(0,h)$ is bounded, then it intersects $E$.
\end{lemma}
\begin{proof}
When $h=0$, $S(0, h) = \{0\}$ and the result is immediate. For $h \not = 0$, fix any $r > 0$. Boundedness of $S(0,h)$ implies boundedness of $D(0, h)$. Given any $y \in Y$, the ray $r_{0 \to y}$ is unbounded and lies in $Y$, and so it must intersect $S(0,h)$. Similarly boundedness of $\Delta(0, g, r)$ implies that $r_{0 \to y}$ must intersect $P(0,g, r)$. Both $\langle g, \cdot \rangle$ and $H$ increase strictly along $r_{0 \to y}$, and so $r_{0 \to y}$ intersects each of $P(0,g, r)$ and $S(0,h)$ exactly once. Both $P(0,g, r)$ and $S(0,h)$ are compact, and so by Lemma~\ref{twosets}, they are homeomorphic. By Lemma~\ref{cc1}, $P(0, g, r)$ is an $n-1$ dimensional ball, and thus so is $S(0,h)$. As $S(0,h)$ is also forward invariant, by well known results it contains an equilibrium (see for example Thm~12, p197 in \cite{spanier}). 
\end{proof}\\

Lemmas~\ref{lemcdash}~to~\ref{mainorderlemma} below are all leading towards Lemma~\ref{mainhomeolem}. For all of these lemmas, we fix some $c \in Y$ and some constants $k_1$, $k_2$. There are two cases which will be referred to as {\bf Case 1} and {\bf Case 2}:\\
{\bf Case 1.} $\langle g, c\rangle < k_1 < k_2$,\\
{\bf Case 2.} $c \not = 0$ and $0\leq k_2 < k_1 < \langle g, c\rangle$.

To shorten notation, define $P_1 \equiv P(c, g, k_1)$, $P_2 \equiv P(c, g, k_2)$, $\Delta_1 \equiv \Delta(c, g, k_1)$, $\Delta_2 \equiv \Delta(c, g, k_2)$. Lemmas~\ref{smalldelta}~and~\ref{smalldelta1} tell us that:
\begin{enumerate}
\item $\Delta_1$ and $\Delta_2 \supset \Delta_1$ are nonempty, compact, convex regions. 
\item In Case 1, $\mathrm{ri}(P_1) = P_1 \cap \mathrm{int}(c^{+})$ and $\mathrm{relbd}(P_1) = P_1 \cap \partial c^{+}$, while in Case 2, $\mathrm{ri}(P_1) = P_1 \cap \mathrm{ri}(c^{-})$ and $\mathrm{relbd}(P_1) = P_1 \cap \mathrm{relbd}(c^{-})$. Provided that $k_2 > 0$, $\mathrm{relbd}(P_2)$ and $\mathrm{ri}(P_2)$ are similarly characterised. 
\item In Case 1, $\partial \Delta_1$ is the disjoint union of $\Delta_1 \cap \partial c^{+}$ and $P_1 \cap \mathrm{int}(c^{+})$, while in Case 2, $\mathrm{relbd}(\Delta_1)$ is the disjoint union of $\Delta_1 \cap \mathrm{relbd}(c^{-})$ and $P_1 \cap \mathrm{ri}(c^{-})$. Similar statements apply to $\mathrm{relbd}(\Delta_2)$ (when $k_2 = 0$, $P_2 \cap \mathrm{ri}(c^{-})$ is empty, and $\mathrm{relbd}(\Delta_2) = \Delta_2 \cap \mathrm{relbd}(c^{-})$).
\end{enumerate}

Define
\[
\Theta \equiv \left\{ \begin{array}{ll}\inf_{x \in P_1, y \in \Delta_2}\langle x-c, \nabla H(y) \rangle, & \mbox{Case 1}\\
\inf_{x \in P_1, y \in \Delta_2}\langle c-x, \nabla H(y) \rangle, & \mbox{Case 2}\,.\end{array} \right .
\]
Since for each $y \in Y$, $\nabla H(y) \in \mathrm{int}(K^*)$, we know that for any $x \in (c+K)\backslash\{c\}$, $\langle x-c, \nabla H(y)\rangle > 0$, and for each $x \in (c-K)\backslash\{c\}$, $\langle c-x, \nabla H(y)\rangle > 0$. Thus in each case, since $P_1$ and $\Delta_2$ are compact, $\Theta > 0$.

\vspace{0.25cm}
\begin{lemma}
\label{lemcdash}
There exists $s_0 \in \mathrm{ri}(\Delta_1)$ such that \\
{\bf Case 1.} $\langle s_0 - c, \nabla H(y) \rangle < \Theta$ for all $y \in \Delta_2$, and for any $x \in P_1$, $H$ increases strictly along $r_{s_0 \to x}$ within $\Delta_2$.\\
{\bf Case 2.} $\langle c - s_0, \nabla H(y) \rangle < \Theta$ for all $y \in \Delta_2$, and for any $x \in P_1$, $H$ decreases strictly along $r_{s_0 \to x}$ within $\Delta_2$.
\end{lemma}
\begin{proof}
Define $\nabla H_{max} \equiv \sup_{y \in \Delta_2}|\nabla H(y)|$. Since $\nabla H(y) \not = 0$, and moreover $\nabla H(y)$ is continuous and $\Delta_2$ compact, $0 < \nabla H_{max} < \infty$. Choose $\epsilon > 0$ such that
\[
\epsilon < \min\{\Theta/\nabla H_{max}, |k_1 - \langle g, c\rangle|\}\,.
\]
Given $x \in P_1$, $y, z \in \Delta_2$, and $t_1 < t_2$ such that $x_1 \equiv z + t_1(x-z)$ and $x_2 \equiv z + t_2(x-z)$ lie in $\Delta_2$,
\begin{equation}
\label{eqinc}
\langle x_2 - x_1, \nabla H(y) \rangle = (t_2-t_1)\left(\langle x-c, \nabla H(y) \rangle - \langle z - c, \nabla H(y) \rangle \right).
\end{equation}
{\bf Case 1.} Choose any $y^{'} \in \mathrm{int}(Y)$ such that $|y^{'}| = \epsilon$, and set $s_0 = c + y^{'}$ so that $|s_0 - c| = \epsilon$. By {\bf O1}, $s_0 \in \mathrm{int}(Y)$; $s_0 \in c + \mathrm{int}(K)$ since $y^{'} \in \mathrm{int}(K)$; and $\langle g, s_0\rangle \leq \langle g, c\rangle + \epsilon < k_1$. So $s_0 \in \mathrm{int}(\Delta_1)$. In addition, for any $y \in \Delta_2$, 
\[
\langle s_0 - c,  \nabla H(y) \rangle \leq \epsilon\,|\nabla H(y)| < \Theta\,.
\]
Setting $z = s_0$ in (\ref{eqinc}), $\langle x_2 - x_1, \nabla H(y) \rangle > 0$.\\
{\bf Case 2.} Choose $s_0 \in \mathrm{ri}(c^{-})$ so that $|s_0 - c| \leq \epsilon$. As 
\[
\langle g, s_0\rangle = \langle g, c\rangle - \langle g, c - s_0\rangle  \geq \langle g, c\rangle - \epsilon > k_1,
\]
so $s_0 \in \mathrm{ri}(\Delta_1)$. In addition, for any $y \in \Delta_2$, 
\[
\langle c-s_0,  \nabla H(y) \rangle \leq \epsilon\,|\nabla H(y)| < \Theta\,.
\]
Setting $z = s_0$ in (\ref{eqinc}), $\langle x_2 - x_1, \nabla H(y) \rangle < 0$.

This completes the proof.
\end{proof}\\

\begin{lemma}
\label{interiorray}
With $s_0$ defined as in Lemma~\ref{lemcdash}:\\
{\bf Case 1.} $r_{s_0 \to 2s_0}$ intersects both $\mathrm{ri}(P_1)$ and $\mathrm{ri}(P_2)$.\\
{\bf Case 2.} $r_{s_0 \to 0}$ intersects both $\mathrm{ri}(P_1)$ and $\mathrm{ri}(P_2)$. \\
In each case, the points of intersection are unique. 
\end{lemma}
\begin{proof}
{\bf Case 1.} Since $s_0 \in \mathrm{int}(Y)$, $ts_0 \in \mathrm{int}(Y)$ for all $t > 0$. For $t \geq 1$, $ts_0 = c + (t-1)c + t(s_0- c)  \in c + \mathrm{int}(K)$, since $(t-1)c \in K$ and $s_0- c \in \mathrm{int}(K)$. So $r_{s_0 \to 2s_0} \subset \mathrm{int}(c^{+})$. As $r_{s_0 \to 2s_0}$ is unbounded, and hence leaves $\Delta_1$, there must exist some $t_1 > 1$ such that $t_1s_0 \in \mathrm{ri}(P_1)$. Applying a similar argument to $\Delta_2$, there must be some $t_2 > t_1$ such that $t_2s_0 \in \mathrm{ri}(P_2)$. Uniqueness of the point of intersection follows from the fact that $\langle g, \cdot \rangle$ increases strictly along $r_{s_0 \to 2s_0}$.

{\bf Case 2.} $r_{s_0 \to 0}$ exits $c^{-}$ at $0$. Since $0 \not \in\Delta_1$, there must exist some $0 < t_1 < 1$ such that $t_1s_0 \in \mathrm{ri}(P_1)$. If $k_2 > 0$, applying a similar argument to $\Delta_2$, there must be some $0 < t_2 < t_1$ such that $t_2s_0 \in \mathrm{ri}(P_2)$. If $k_2 = 0$, then $P_2 = \mathrm{ri}(P_2) = \{0\}$ and $r_{s_0 \to 0}$ intersects $P_2$ at this point. Uniqueness follows since $\langle g, \cdot \rangle$ decreases strictly along $r_{s_0 \to 0}$.
\end{proof}\\

From now on, given an arbitrary but fixed $c \in Y$, $s_0$ will refer to some point defined as in Lemma~\ref{lemcdash}, and following Lemma~\ref{interiorray}, we define $s_1= r_{0 \to s_0} \cap P_1$, $s_2= r_{0 \to s_0} \cap P_2$, and $t_1, t_2$ by $s_1 = t_1s_0$ and $s_2= t_2s_0$. These definitions are illustrated schematically in Figure~\ref{basic2}.

\begin{figure}[h]
\begin{minipage}[h]{0.4\textwidth}
\begin{center}
\includegraphics[width=0.9\textwidth]{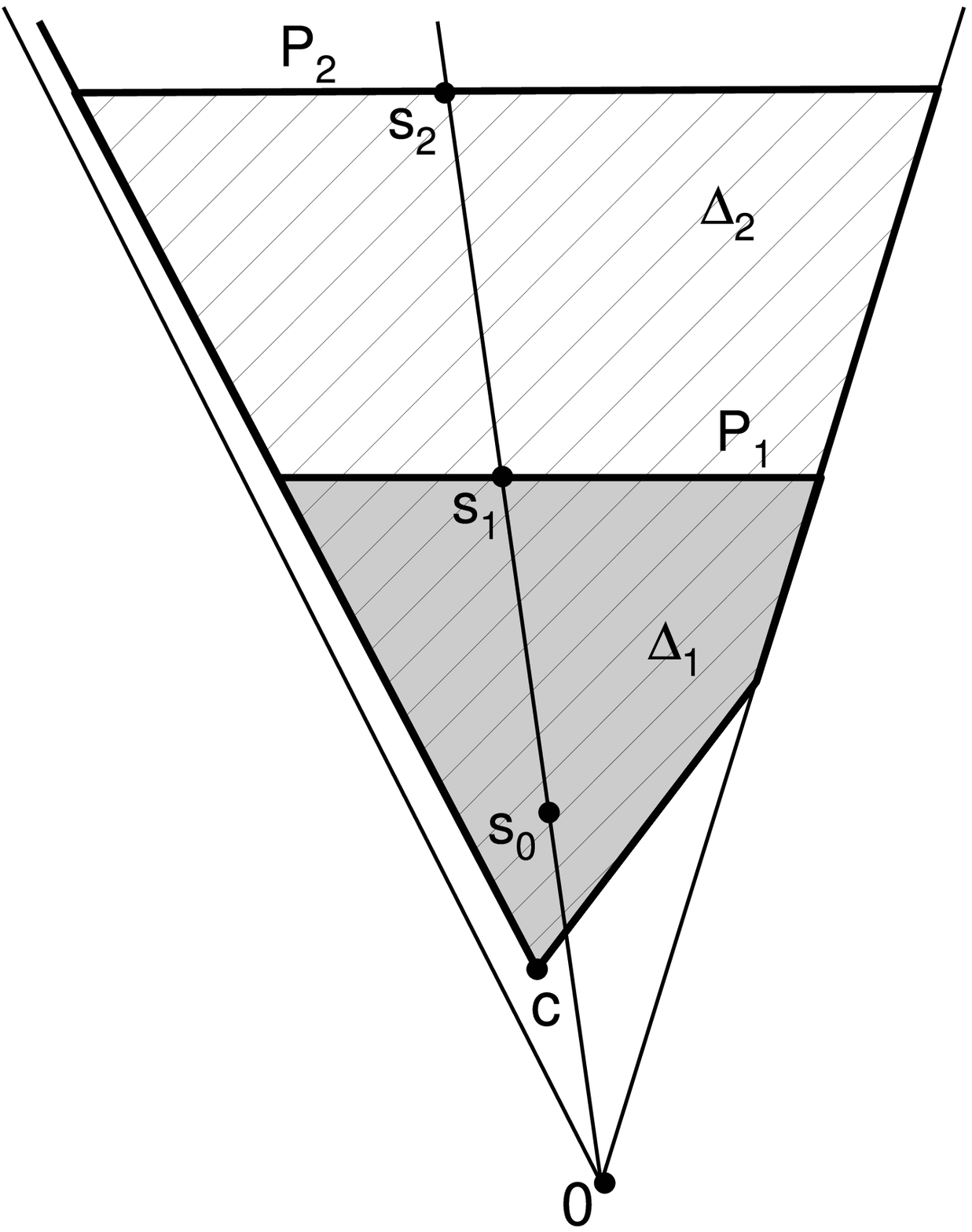}
\end{center}
\end{minipage}
\hfill
\begin{minipage}[h]{0.4\textwidth}
\begin{center}
\includegraphics[width=0.9\textwidth]{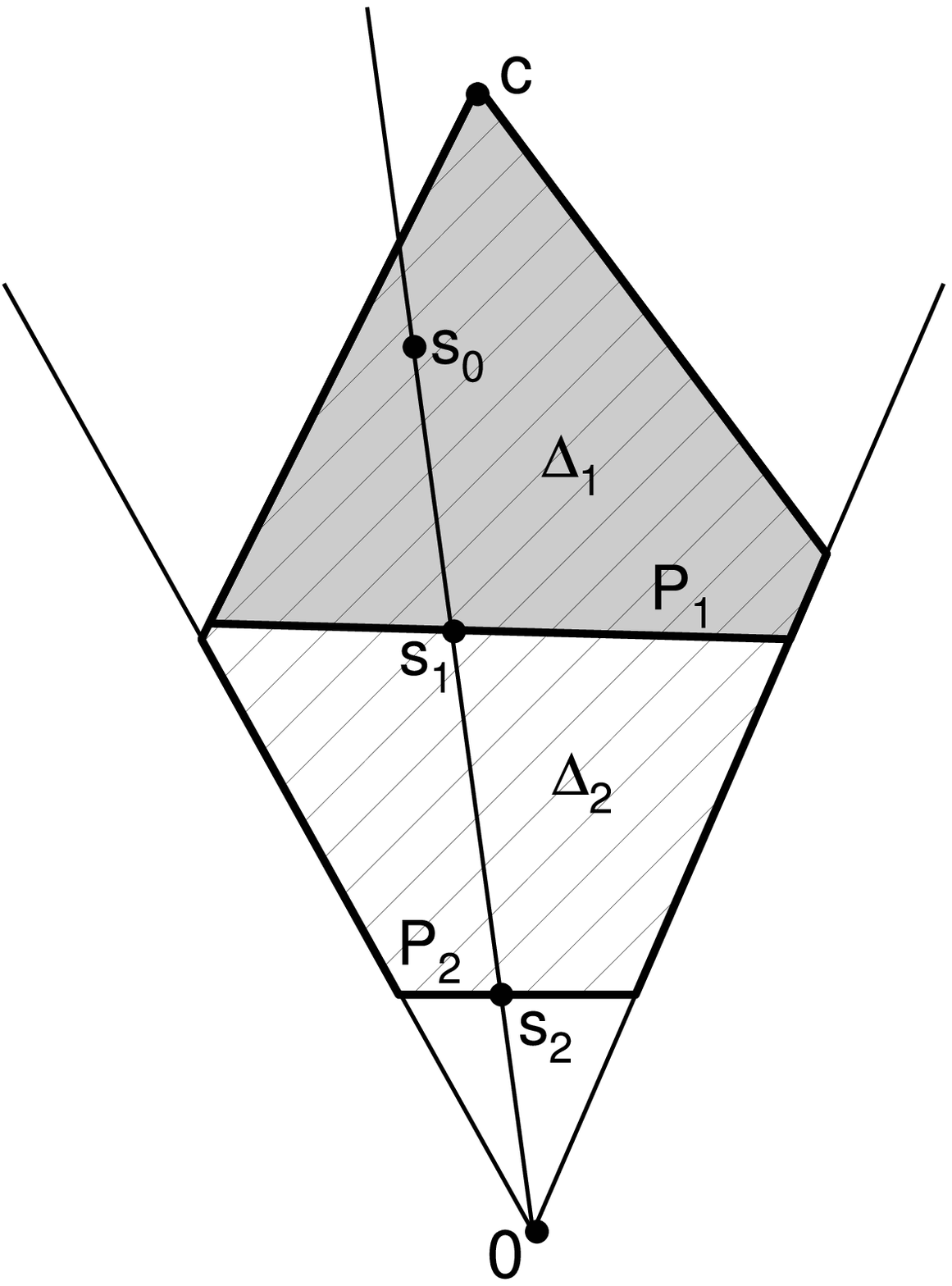}
\end{center}
\end{minipage}
\caption{\label{basic2}The ray $r_{0 \to s_0}$ intersects both $P_1$ and $P_2$ at points $s_1$ and $s_2$ respectively. {\em Left.} The situation when $s_0  \in c + K$. {\em Right.} The situation when $s_0 \in c - K$.}
\end{figure}

\begin{lemma}
\label{D2lem}
Assume $k_2 > 0$. Given any nonzero $\delta \in T(c,g,k_2)$, there exists $t(\delta)>0$ such that i) for $t \in [0, t(\delta))$, $r_{s_0 \to (s_1+t\delta)}$ intersects $\mathrm{ri}(P_2)$; ii) $r_{s_0 \to (s_1+t(\delta)\delta)}$ intersects $\mathrm{relbd}(P_2)$; iii) for $t > t(\delta)$, $r_{s_0 \to (s_1+t\delta)}$ does not intersect $P_2$.
\end{lemma}
\begin{proof}
The assumption on $\delta$ implies that $s_1 + t\delta \in \mathrm{Aff}(P_1)$ for all $t$, and since $k_2 > 0$, $s_2 + t\delta \in \mathrm{Aff}(P_2)$ for all $t$.
A quick calculation reveals that the ray $r_{s_0 \to s_1+t\delta}$ intersects $P(g, k_2)$ (the hyperplane containing $P_2$) at the point $s_2 + t\delta(t_2-1)/(t_1-1)$, i.e. on $r_{s_2 \to s_2 + \delta}$. 
Moreover, $t(t_2-1)/(t_1-1)$ is an increasing function of $t$. By Lemma~\ref{interiorray}, $s_2 \in \mathrm{ri}(P_2)$, and so by {\bf O2}, there is exactly one point where $r_{s_2 \to s_2 + \delta}$ intersects $\mathrm{relbd}(P_2)$, and thus exactly one value $t(\delta)>0$ such that $r_{s_0 \to (s_1+t(\delta)\delta)}$ intersects $\mathrm{relbd}(P_2)$ with $r_{s_0 \to (s_1+t\delta)}$ intersecting $\mathrm{ri}(P_2)$ for $t \in [0, t(\delta))$, and failing to intersect $P_2$ for $t > t(\delta)$.
\end{proof}\\

\begin{lemma}
\label{corD2}
If, for some nonzero $\delta \in T(c,g,k_2)$, and some $t^{'} > 1$, $r_{s_0 \to s_1+\delta}$ and $r_{s_0 \to s_1+t^{'}\delta}$ do not intersect $P_2 \cap \mathrm{int}(c^{+})$ (Case 1) or $P_2 \cap \mathrm{ri}(c^{-})$ (Case 2), then for all $t \in [1, t^{'}]$, $r_{s_0 \to s_1+t\delta}$ do not intersect $P_2 \cap \mathrm{int}(c^{+})$ (Case 1) or $P_2 \cap \mathrm{ri}(c^{-})$ (Case 2).
\end{lemma}
\begin{proof}
When $k_2 = 0$, then $P_2 \cap \mathrm{ri}(c^{-})$ is empty and the result is immediate. Assume $k_2 > 0$, in which case, in Case 1, $P_2 \cap \mathrm{int}(c^{+}) = \mathrm{ri}(P_2)$, while in Case 2, $P_2 \cap \mathrm{ri}(c^{-}) = \mathrm{ri}(P_2)$. If there is some $t_0 \in [1, t^{'}]$ such that $r_{s_0 \to s_1+t_0\delta}$ intersects $\mathrm{ri}(P_2)$, then by Lemma~\ref{D2lem}, for $t \in [0, t_0]$, $r_{s_0 \to s_1+t\delta}$ must intersect $\mathrm{ri}(P_2)$, contradicting the fact that $r_{s_0 \to s_1+\delta}$ does not intersect $\mathrm{ri}(P_2)$.
\end{proof}\\

\begin{lemma}
\label{case2}
Let $0 \not = \delta \in g^\perp$, $t > 1$, $y \in Y$ and $s = ky$ for some $k \in (0, 1) \cup (1, \infty)$. Suppose $r_{y \to (s+\delta)}$ and $r_{y \to (s+t\delta)}$ both exit $Y$ at points $p_1 \not = y$ and $p_2 \not = y$ respectively. Then $p_2 = qp_1$, where if $k > 1$, then $0<q < 1$, and if $k < 1$, then $q > 1$.
\end{lemma}
\begin{proof}
Consider the 2D affine subspace spanned by $y$ and $\delta$, and let $Y^{'}$ be the intersection of this subspace with $Y$. $Y^{'}$ is itself a closed, convex and pointed 2D cone, 
and $y, s, p_1, p_2 \in Y^{'}$. Define $\lambda_1,\lambda_2 > 0$ by
\[
p_1 = y + \lambda_1(s-y + \delta),\quad p_2 = y + \lambda_2(s-y + t\delta).
\]
By assumption, $p_1, p_2 \in \mathrm{relbd}(Y^{'})$. Further, define $\lambda_2^{'}$ and $q$ by
\[
\lambda_2^{'} = \frac{\lambda_1}{\lambda_1(k-1)(t-1) + t},\quad q = \frac{t}{\lambda_1(k-1)(t-1) + t}.
\]
By observation, if $k > 1$, then $\lambda_2^{'} > 0$ and $0<q < 1$. If $k\in (0, 1)$, then
\[
y + \frac{1}{1-k}(s-y + \delta) = y + \frac{1}{1-k}((k-1)y + \delta) = \frac{\delta}{1-k} \not \in Y\,,
\]
so $\lambda_1 < \frac{1}{1-k}$ and $-1 < \lambda_1(k-1) < 0$, and thus $1 < \lambda_1(k-1)(t-1) + t < t$, implying that $\lambda_2^{'} > 0$ and $q > 1$. Further, a quick calculation reveals that $p_2^{'} \equiv y + \lambda_2^{'}(s-y + t\delta) = qp_1$. 
But $p_2^{'} \in Y$ since $p_1$ in $Y$, and $p_2^{'} \in \mathrm{relbd}(Y^{'})$ since $p_1 \in \mathrm{relbd}(Y^{'})$. Moreover, by {\bf O2}, the intersection between $r_{y \to (s + t\delta)}\backslash\{y\}$ and $\mathrm{relbd}(Y^{'})$ is unique, and so $p_2^{'} = p_2$. The two cases are illustrated in Figure~\ref{fig1}. 
\end{proof}\\

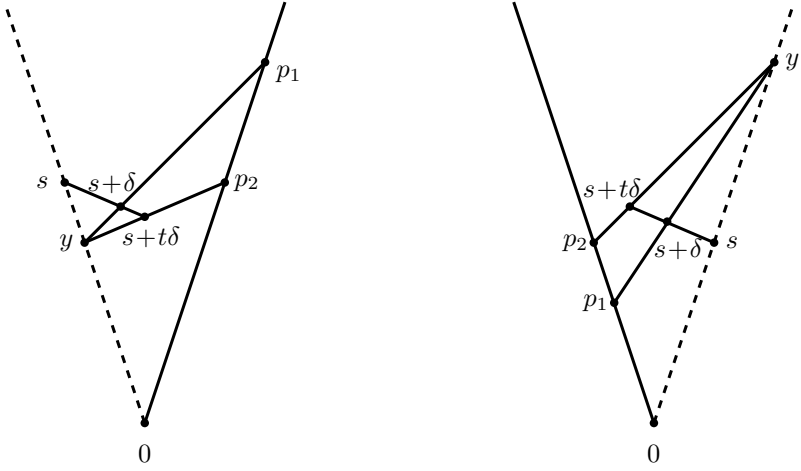
\begin{figure}[h]
\begin{minipage}{0.48\textwidth}
\begin{tikzpicture}[domain=-6:6,scale=0.8]
\path (0,0) coordinate (origin);
\path (2.33, 7) coordinate (xline);
\path (-2.33, 7) coordinate (yline);
\path (-1, 3) coordinate (c);
\path (-1.33, 4) coordinate (tc);
\path (1.33, 4) coordinate (bx);
\path (2, 6) coordinate (bx1);

\path (1, 3) coordinate (cc);

\fill (origin) circle (2pt);
\fill (c) circle (2pt);
\fill (tc) circle (2pt);
\fill (bx) circle (2pt);
\fill (bx1) circle (2pt);
\draw[-, dashed, line width=0.04cm] (origin) -- (yline);
\draw[-, line width=0.04cm] (origin) -- (xline);
\draw[-, line width=0.04cm] (c) -- (bx);
\draw[-, line width=0.04cm] (c) -- (bx1);
\draw[-, line width=0.04cm] (tc) -- (0, 3.43);

\fill (-0.4, 3.6) circle (2pt);
\fill (0, 3.43) circle (2pt);

\node at (0, -0.5) {$0$};
\node at (-1.3, 3) {$y$};
\node at (-1.7, 4) {$s$};
\node at (1.7, 4) {$p_2$};
\node at (2.4, 5.8) {$p_1$};
\node at (-0.55, 4.0) {$s\!+\!\delta$};
\node at (0.1, 3.1) {$s\!+\!t\delta$};
\end{tikzpicture}
\end{minipage}
\hfill
\begin{minipage}{0.48\textwidth}
\begin{tikzpicture}[domain=-6:6,scale=0.8]
\path (0,0) coordinate (origin);
\path (2.33, 7) coordinate (xline);
\path (-2.33, 7) coordinate (yline);
\path (-1, 3) coordinate (c);
\path (-0.66, 2) coordinate (cc1);
\path (2, 6) coordinate (bx1);

\path (1, 3) coordinate (s);

\fill (origin) circle (2pt);
\fill (c) circle (2pt);
\fill (cc1) circle (2pt);

\fill (bx1) circle (2pt);
\fill (s) circle (2pt);
\draw[-, line width=0.04cm] (origin) -- (yline);
\draw[-, dashed, line width=0.04cm] (origin) -- (xline);
\draw[-, line width=0.04cm] (bx1) -- (c);
\draw[-, line width=0.04cm] (bx1) -- (cc1);
\draw[-, line width=0.04cm] (1, 3) -- (-0.4, 3.6);

\fill (-0.4, 3.6) circle (2pt);
\fill (0.22, 3.35) circle (2pt);

\node at (0, -0.5) {$0$};
\node at (-1.3, 3) {$p_2$};
\node at (-0.96, 2) {$p_1$};
\node at (1.3, 3) {$s$};
\node at (2.3, 6) {$y$};
\node at (-0.7, 3.9) {$s\!+\!t\delta$};
\node at (0.4, 2.9) {$s\!+\!\delta$};
\end{tikzpicture}
\end{minipage}
\caption{\label{fig1}An illustration of the two situations in Lemma~\ref{case2}. In both cases, $p_1$ and $p_2$ are collinear with the origin. {\em Left.} $s = ky$ where $k > 1$: in this case, $p_1 > p_2$. {\em Right.} $s = ky$ where $0 < k < 1$: in this case, $p_2 > p_1$.}
\end{figure}

\begin{lemma}
\label{case1}
Consider some $c \in Y$, nonzero $\delta \in g^\perp$ and $t > 1$.\\
{\bf Case 1.} Let $y \in (c + K)\backslash\{c\}$ and let $s = ky$ for some $k > 1$.\\
{\bf Case 2.} Let $y \in (c - K)\backslash\{c\}$ and let $s = ky$ for some $k \in (0, 1)$.\\
Suppose $r_{y \to (s+\delta)}$ and $r_{y \to (s+t\delta)}$ both exit $c+K$ (Case 1) or $c-K$ (Case 2) at points $p_1\not = y$ and $p_2\not = y$ respectively. Then $p_1 - p_2 =  R_0\left((k-1)y + R\delta \right)$, where $R_0 > 0$ and $0 \leq R < 1$. 
\end{lemma}
\begin{proof}
Note that $s-y = (k - 1)y$, and define $\lambda_1, \lambda_2 > 0$ by
\[
p_1 = y + \lambda_1((k - 1)y + \delta),\quad p_2 = y + \lambda_2((k - 1)y + t\delta).
\]
$p_1, p_2 \not = 0$ since $\delta \not = 0$. Consider the point
\[
z \equiv y+ \lambda_2t((k - 1)y + \delta) = p_2 + \lambda_2(t-1)(k - 1)y.
\]
In Case 1, as $p_2 \in c + K$ and $(k - 1)y \in K$, $z \in c + K$. In Case 2, as $p_2 \in c - K$ and $-(k - 1)y \in K$, $z \in c - K$. In each case, this implies, by {\bf O2}, that $\lambda_1 \geq \lambda_2t > \lambda_2$, from which we get $0 < \frac{\lambda_2 (t-1)}{\lambda_1 - \lambda_2} \leq 1$. Define $R_0 \equiv \lambda_1 - \lambda_2 > 0$ and $R \equiv 1-\frac{\lambda_2 (t-1)}{\lambda_1 - \lambda_2}$, and note that $0\leq R<1$. Now we calculate $p_1-p_2$:
\[
p_1 - p_2 =  (\lambda_1-\lambda_2)\left((k - 1)y + \delta - \frac{\lambda_2 (t-1)}{(\lambda_1-\lambda_2)} \delta\right) =  R_0\left((k - 1)y + R\delta \right).
\]
\end{proof}\\

We now come to a key lemma:

\vspace{0.25cm}
\begin{lemma}
\label{mainorderlemma}
Consider some nonzero $\delta \in T(c,g,k_1)$ and $t_0 > 1$, such that $s_1+\delta$ and $s_1+t_0\delta$ both lie in $P_1$. 
Define $p_1 \equiv \Pi_{s_0, \mathrm{relbd}(\Delta_2)}(s_1+\delta)$ and $p_2 \equiv \Pi_{s_0, \mathrm{relbd}(\Delta_2)}(s_1+t_0\delta)$, and suppose $p_1, p_2 \in \partial c^{+}$ (Case 1) or $p_1, p_2 \in \mathrm{relbd}(c^{-})$ (Case 2). Then in Case 1, $H(p_1) > H(p_2)$, and in Case 2, $H(p_1) < H(p_2)$. 
\end{lemma}
\begin{proof}
Note that $s_1-s_0 = (t_1 - 1)s_0$ and define $\lambda_1, \lambda_2$ by
\[
p_1 = s_0 + \lambda_1((t_1 - 1)s_0 + \delta),\quad p_2 = s_0 + \lambda_2((t_1 - 1)s_0 + t_0\delta).
\]
By convexity of $\Delta_2$, $\lambda_1, \lambda_2 \geq 1$. In Case 2, define $Y_c \equiv Y \cap \mathrm{Aff}(c^{-})$, and $K_c \equiv K \cap \mathrm{Aff}(c^{-})$. The reader is reminded of the remarks following Lemma~\ref{smalldelta1}. 

{\bf Possibility 1.} $p_1, p_2 \in c+\partial K$ (Case 1), or $p_1, p_2 \in c-\mathrm{relbd}(K_c)$ (Case 2). By Lemma~\ref{case1}, $p_1 - p_2 = R_0\left((t_1 - 1)s_0 + R\delta \right)$, where $R_0 > 0$ and $0 \leq R < 1$. So for any $y \in \Delta_2$, $\langle p_1-p_2, \nabla H(y) \rangle = R_0\langle s_1 + R\delta - s_0, \nabla H(y) \rangle$. Since $0\leq R<1$, $s_1 + R\delta \in P_1$. In Case 1, by Lemma~\ref{lemcdash}, $\langle (s_1 + R\delta) - s_0, \nabla H(y) \rangle > 0$, and so $\langle p_1-p_2, \nabla H(y) \rangle > 0$, and consequently $H(p_1) > H(p_2)$. In Case 2, by Lemma~\ref{lemcdash}, $\langle s_0 - (s_1 + R\delta), \nabla H(y) \rangle > 0$, and so $\langle p_1-p_2, \nabla H(y) \rangle < 0$, and consequently $H(p_1) < H(p_2)$.

{\bf Possibility 2.}  $p_1, p_2 \in \partial Y$ (Case 1) or $p_1, p_2 \in \mathrm{relbd}(Y_c)$ (Case 2). In this case, Lemma~\ref{case2} implies that $p_2 = qp_1$, where, in Case 1, $0<q < 1$, and hence $H(p_1) > H(p_2)$, and in Case 2, $q > 1$ and hence $H(p_2) > H(p_1)$.

{\bf Possibility 3.} Assume that neither Possibility 1 nor Possibility 2 holds. Define $b(t) = \Pi_{s_0, \mathrm{relbd}(\Delta_2)}(s_1+t\delta)$. Since $b(1) = p_1$ and $b(t_0) = p_2$ lie in $\partial c^{+}$ (Case 1) or in $\mathrm{relbd}(c^{-})$ (Case 2), by Lemma~\ref{corD2}, for each $t \in [1, t_0]$, $b(t)$ lies in $\partial c^{+}$ (Case 1) or in $\mathrm{relbd}(c^{-})$ (Case 2). Since both $P_1$ and $\mathrm{relbd}(\Delta_2)$ are compact and disjoint from $s_0$, and each ray $r_{s_0 \to (s_1+t\delta)}$ intersects each exactly once, by Lemma~\ref{twosets}, $B \equiv \{b(t)\,|\, t \in [1,t_0]\}$ is homeomorphic to a closed interval, and hence closed and connected. Define $B_1 \equiv B \cap \partial Y$, $B_2 \equiv B \cap (c + \partial K)$ (Case 1), or $B_1 \equiv B \cap \mathrm{relbd}(Y_c)$, $B_2 \equiv B \cap (c - \mathrm{relbd}(K_c))$ (Case 2). Both $B_1$ and $B_2$ are closed nonempty sets, so there must exist $p_3 \in (B_1 \cap B_2)$, (otherwise $B \backslash B_1$ and $B \backslash B_2$ would form a separation of $B$). Note that $p_3 = b(t_3)$ for some $t_3 \in (1, t_0)$. Then, from Possibilities 1 and 2, in Case 1, $H(p_2) < H(p_3) < H(p_1)$ implying that $H(p_2) < H(p_1)$, and in Case 2, $H(p_2) > H(p_3) > H(p_1)$ implying that $H(p_2) > H(p_1)$. 
\end{proof}\\

Define $\tilde H(x)$ on $\Delta_2\backslash\{s_0\}$ by $\tilde H(x) \equiv H(\Pi_{s_0, \mathrm{relbd}(\Delta_2)}(x))$. Since $\Pi_{s_0, \mathrm{relbd}(\Delta_2)}$ is continuous by Lemma~\ref{convexproj}, $\tilde H(x)$ is continuous as the composition of continuous functions. 

\vspace{0.25cm}
\begin{lemma}
\label{mainhomeolem}
Consider any  $c \in Y$. There is some $\epsilon_c > 0$ such that for $h \in [0, H(c) + \epsilon_c)$, $S(c, h)$ is a ball. 
\end{lemma}
\begin{proof}
In the special cases $h=0$ and $h = H(c)$, $S(c, h)$ is a single point, and hence automatically a $0$ dimensional ball. We next treat the case $h > H(c)$, and use the construction defined at the end of Section~\ref{secprelim} and illustrated in Figure~\ref{stages}. Choose some $k_2 > \langle g, c\rangle$ and as before, define $P_2 \equiv (c, g, k_2)$ and $\Delta_2 \equiv \Delta(c, g, k_2)$. Define $h_{min} \equiv \min\{H(x)\,|\,x \in P_2\} > H(c)$, $\epsilon_c \equiv H(c) - h_{min}$, and choose any $h \in (H(c), H(c) + \epsilon_c)$. By Lemma~\ref{boundedD}, $S(c, h), D(c, h) \subset \Delta_2$. By Lemma~\ref{boundedP}, choose some $k_1$ satisfying $\langle g, c\rangle < k_1 < k_2$ and so that $\mathrm{max}_{x \in \Delta(c, g, k_1)}H(x) < h$. As usual, define $P_1 \equiv P(c, g, k_1)$ and $\Delta_1 \equiv \Delta(c, g, k_1)$. By Lemma~\ref{boundedP}, $\Delta_1 \subset D(c, h)$. 

With $\Theta$ defined as above, by Lemma~\ref{lemcdash} there exists $s_0 \in \mathrm{int}(\Delta_1)$ such that $\langle s_0 - c, \nabla H(y) \rangle < \Theta$. Consider the projection $\Pi_{s_0, P_1}$ onto $P_1$. By the arguments in Section~\ref{secprelim}, $S(c, h) \subset K(s_0, P_1)$. 
By Lemma~\ref{twosets}, $S(c, h)$ is homeomorphic to $P_{10} \subset P_1$ where $P_{10} = \Pi_{s_0, P_1}(S(c, h))$. The construction is illustrated in Figure~\ref{schematic} (left). Since, by Lemma~\ref{lemcdash}, $H$ increases in $\Delta_2$ along rays $r_{s_0 \to x}$ (where $x \in P_1$), we can also characterise $P_{10}$ by $P_{10} = \{x \in P_1\,|\, \tilde H(x) \geq h\}$. By continuity of $\tilde H$, we can characterise $\mathrm{ri}(P_{10}) = \{x \in P_1\,|\, \tilde H(x) > h\}$.

\begin{figure}[h]
\begin{minipage}[h]{0.44\textwidth}
\begin{center}
\includegraphics[width=0.9\textwidth]{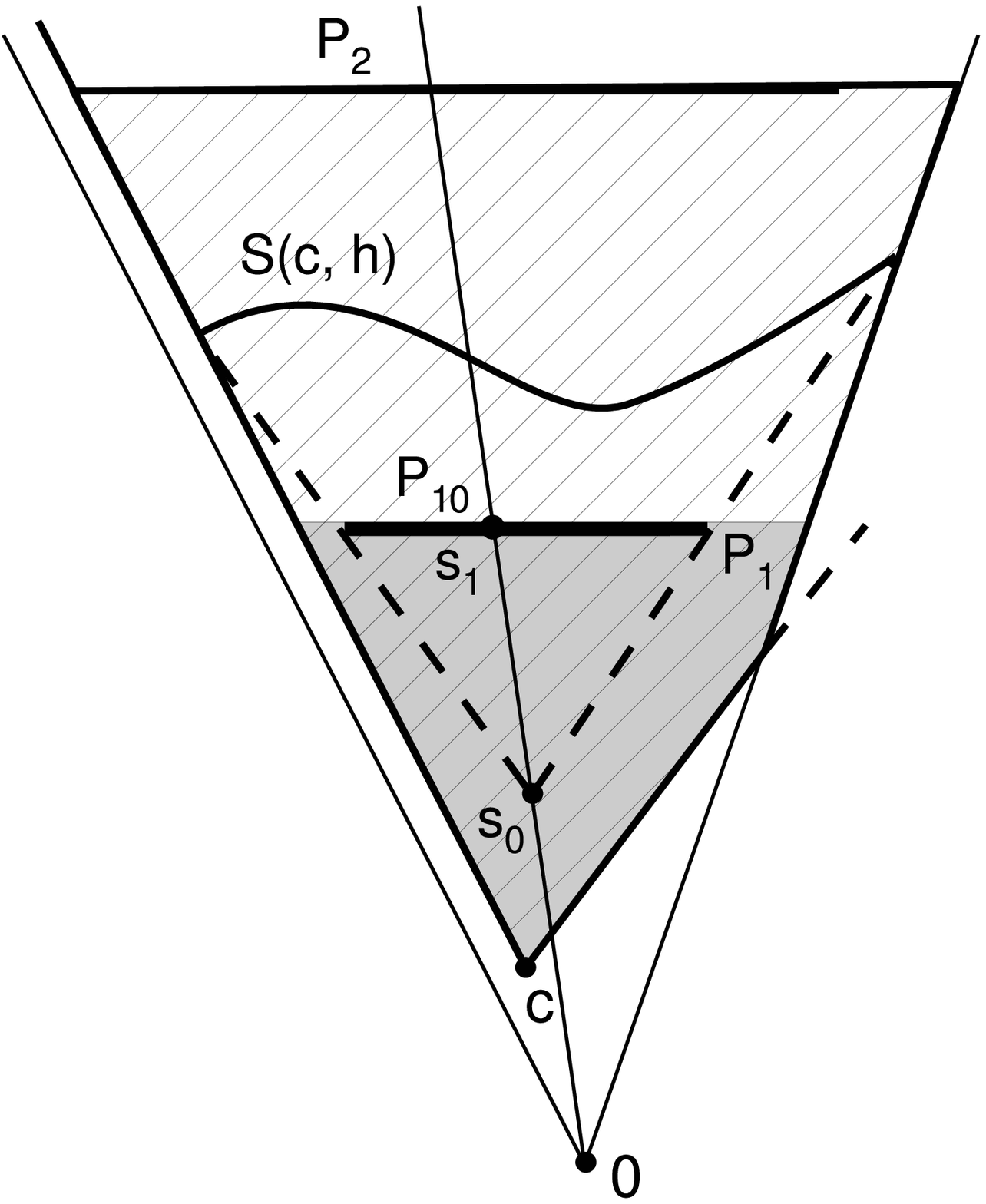}
\end{center}
\end{minipage}
\hfill
\begin{minipage}[h]{0.4\textwidth}
\begin{center}
\includegraphics[width=0.9\textwidth]{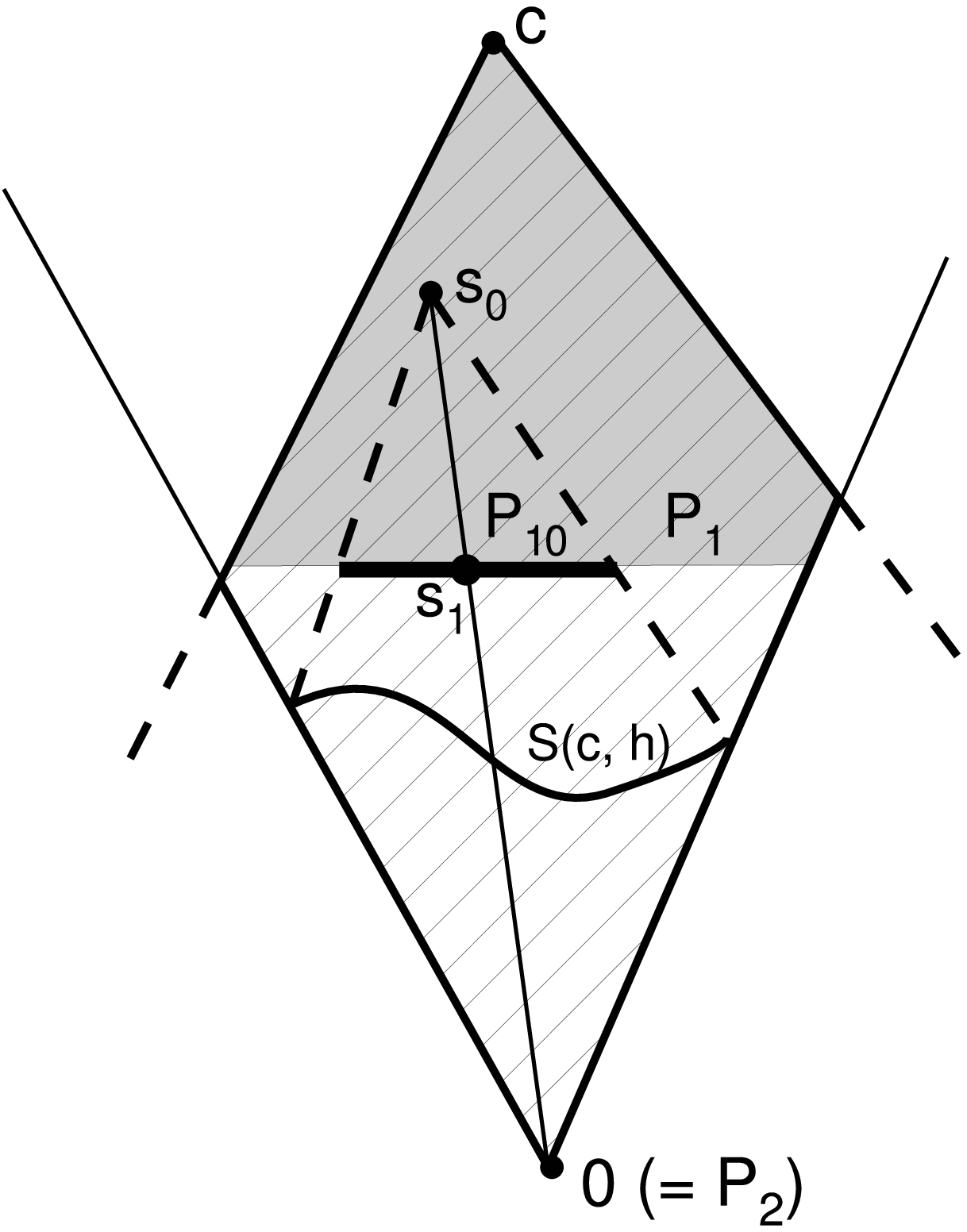}
\end{center}
\end{minipage}
\caption{\label{schematic}The construction of the set $P_{10}$ homeomorphic to $S(c,h)$. $P_1$ and $P_2$ ``enclose'' $S(c, h)$. $s_0$ lies in the interior of $\Delta_1$ (shaded region). $\Delta_2$ (hatched region) is compact. Each ray emanating from $s_0$ which intersects $S(c, h)$ intersects $P_1$ at a unique point. $P_{10}$ can thus be defined as $P_{10} = \Pi_{s_0, P_1}(S(c, h))$. {\em Left.} The situation where $h > H(c)$. {\em Right.} The situation where $h < H(c)$. In this case, $P_2 = \{0\}$.}
\end{figure}

Choose any nonzero $\delta \in g^\perp$. We now show that each ray $r_{s_1 \to \delta}$ intersects $\mathrm{relbd}(P_{10})$ exactly once. Let $b(t) = \Pi_{s_0, \mathrm{relbd}(\Delta_2)}(s_1 + t\delta)$ and $G(t) = \tilde H(s_1 + t\delta) = H(b(t))$ which is continuous as the composition of continuous functions. If $b(t) \in P_2$, then $G(t) > h$, and so, by Lemma~\ref{interiorray}, $G(0) > h$. At the same time, by {\bf O2}, there is a unique $t_f > 0$ such that $s_1 + t_f\delta \in \mathrm{relbd}(P_1)$, and so $b(t_f) \in P_1$, implying that $G(t_f) < h$. By the intermediate value theorem, there is a value $t^{'} \in (0, t_f)$ such that $G(t^{'}) = h$. Moreover, this value of $t$ is unique: suppose there are two values $t^{'} < t^{''}$ such that $G(t^{'}) = G(t^{''}) = h$. Then $b(t^{'}), b(t^{''}) \not \in P_2$ since $\min_{x \in P_2}(H(x)) > h$, i.e. $b(t^{'}), b(t^{''}) \in \partial c^{+}$. But this contradicts Lemma~\ref{mainorderlemma}. 
\\

We now treat the case $h \in (0, H(c))$. Fix $h$, let $k_2 = 0$ and define $\Delta_2 \equiv \Delta(c, g, k_2)$ as before. Note that since $P(c, g, k_2) = \{0\}$, $\mathrm{relbd}(\Delta_2)  = \mathrm{relbd}(c^{-})$. By Lemma~\ref{boundedD1}, $S(c, h), D(c, h) \subset \Delta_2$. By Lemma~\ref{boundedP1}, we can choose $k_1$ satisfying $0 < k_1 < \langle g, c\rangle$ and so that $\mathrm{min}_{x \in \Delta(c, g, k_1)}H(x) > h$. As usual, let $P_1 \equiv P(c, g, k_1)$ and $\Delta_1 \equiv \Delta(c, g, k_1)$. By Lemma~\ref{boundedP1}, $\Delta_1 \subset D(c, h)$. 

By Lemma~\ref{lemcdash} there exists $s_0 \in \mathrm{ri}(\Delta_1)$ such that $\langle c- s_0, \nabla H(y) \rangle < \Theta$. Consider the projection $\Pi_{s_0, P_1}$ onto $P_1$. As in the previous case, $S(c, h) \subset K(s_0, P_1)$, and $S(c, h)$ is homeomorphic to $P_{10} = \Pi_{s_0, P_1}(S(c, h))$. The construction is illustrated in Figure~\ref{schematic} (right). Since, by Lemma~\ref{lemcdash}, $H$ decreases in $\Delta_2$ along rays $r_{s_0 \to x}$ ($x \in P_1$), we can also characterise $P_{10}$ by $P_{10} = \{x \in P_1\,|\, \tilde H(x) \leq h\}$. Similarly we have $\mathrm{ri}(P_{10}) = \{x \in P_1\,|\, \tilde H(x) < h\}$. 

Choose any nonzero $\delta \in T(c,g,k_1)$. Again, each ray $r_{s_1 \to \delta}$ intersects $\mathrm{relbd}(P_{10})$ exactly once. Let $b(t)$ and $G(t)$ be defined as before. $G(0) = 0 < h$, and again, there is a unique $t_f > 0$ such that $s_1 + t_f\delta \in \mathrm{relbd}(P_1)$, and so $b(t_f) \in P_1$, implying that $G(t_f) > h$. By the intermediate value theorem, there is a value $t^{'} \in (0, t_f)$ such that $G(t^{'}) = h$. Moreover, this value is unique: suppose there are two distinct values $t^{'}, t^{''}$ such that $G(t^{'}) = G(t^{''}) = h$. Recall, that $\mathrm{relbd}(\Delta_2)  = \mathrm{relbd}(c^{-})$, and so $b(t^{'}), b(t^{''}) \in \mathrm{relbd}(c^{-})$. But this contradicts Lemma~\ref{mainorderlemma}. 
\\

We now complete the argument for both cases. The characterisation $P_{10} = \{x \in P_1\,|\, \tilde H(x) \geq h\}$ (Case 1) and $P_{10} = \{x \in P_1\,|\, \tilde H(x) \leq h\}$ (Case 2) shows us that $P_{10}$ is closed as the inverse image of a closed set under a continuous function. It is bounded as a subset of the bounded set $P_1$, and hence compact. $\Pi_{s_0, \mathrm{relbd}(\Delta_2)}(s_1) \in P_2$ implies in Case 1 that $\tilde H(s_1) > h$, and in Case 2 that $\tilde H(s_1) =0 < h$, and so in either case $s_1 \in \mathrm{ri}(P_{10})$. If $P_1$ consists of the single point $s_1$, then $P_{10}$ consists solely of this point. Otherwise, we have seen that each ray emanating from $s_1$ and lying in $P_1$ intersects $\mathrm{relbd}(P_{10})$ exactly once. Applying Lemma~\ref{lemball}, we see that $P_{10}$ is a ball in $\mathrm{Aff}(P_1)$. Consequently $S(c, h)$ is a ball. 
\end{proof}\\

Note that $S(c, h)$ has dimension $\mathrm{dim}(\mathrm{Aff}(P_1))$. So $\mathrm{dim}(S(c, h)) = n-1$ when $h > H(c)$, and when $h < H(c)$, $0 \leq \mathrm{dim}(S(c, h)) \leq n-1$. Having proved this key lemma, we are now in a position to clarify the structure of the equilibrium set $E$. This is done in Lemmas~\ref{hplusminuslem}~to~\ref{Eunbounded}.

\vspace{0.25cm}
\begin{lemma}
\label{hplusminuslem}
Consider any $c \in E$. There is some $\epsilon_c > 0$ such that for each $h \in [0, H(c) + \epsilon_c)$, $S(c, h)$ contains an equilibrium.
\end{lemma}
\begin{proof}
Choosing $\epsilon_c$ as in Lemma~\ref{mainhomeolem}, this lemma tells us that $S(c, h)$ is a ball. Further, $S(c, h)$ is forward invariant by Lemma~\ref{forwardinvar1}. Thus it contains an equilibrium. 
\end{proof}\\

{\bf Remark.} Note that each step towards the proof of Lemma~\ref{hplusminuslem} has needed only monotonicity rather than strong monotonicity of $\phi$. However strong monotonicity is needed for the next result. \\

\begin{lemma}
\label{orderedeq}
Any two equilibria $c_1$ and $c_2$ must satisfy $c_1 \gg c_2$ or $c_2 \gg c_1$. 
\end{lemma}
\begin{proof}
By Lemma~\ref{nobordereq}, we cannot have $c_1 \in c_2^{\partial}$. So either $c_1$ and $c_2$ are unordered or the conclusion is true. We now show that $c_1$ and $c_2$ cannot be unordered. Suppose the contrary, so that $c_2 \in Y\backslash (c_1^{+} \cup c_1^{-})$. By Lemma~\ref{hplusminuslem} we know that for each $h \in [0, H(c_2))$, $S(c_2, h)$ contains an equilibrium. By Lemma~\ref{boundedD1}, for $H(c_2) - h$ sufficiently small, $S(c_2, h) \subset Y\backslash c_1^{-}$. On the other hand, $S(c_2, 0) = \{0\} \subset c_1^{-}$. Thus 
\[
\tilde h \equiv \inf\{h \,|\,S(c_2, h) \backslash c_1^{-}\,\,\mbox{contains an equilibrium}\}, 
\]
satisfies $\tilde h \in (0, H(c_2))$. For $h \in (\tilde h, H(c_2))$, let $e(h)$ be any equilibrium in $S(c_2, h) \backslash c_1^{-}$. Choose some sequence $h_i \downarrow \tilde h$, and let $e_i = e(h_i)$. As $e_i$ is an infinite sequence in $c_2^{-}$, which by Lemma~\ref{smalldelta1} is compact, it has a convergent subsequence $e_{i_k} \to \tilde{e} \in S(0, \tilde h)$. By closure of $E$, $\tilde e$ is an equilibrium, and since $(e_i) \subset c_2^{-}\backslash c_1^{-}$, $\tilde e \in \mathrm{cl}(c_2^{-} \backslash c_1^{-})$. If $\tilde e \in c_2^{-} \backslash c_1^{-}$, then again by Lemma~\ref{boundedD1} for small enough $\epsilon$, $S(\tilde e, \tilde h - \epsilon) \subset Y \backslash c_1^{-}$, and by Lemma~\ref{hplusminuslem}, there is an equilibrium in $S(\tilde e, \tilde h - \epsilon)$. Since $S(\tilde e, \tilde h - \epsilon) \subset S(c_2, \tilde h - \epsilon)$, this contradicts the definition of $\tilde h$. So $\tilde e \in c_2^{-} \cap (c_1 - \partial K)$. But by Lemma~\ref{nobordereq}, there are no equilibria in $c_1 - \partial K$ except $c_1$, and by assumption $c_1 \not \in c_2^{-}$. 
\end{proof}\\

{\bf Remarks and definitions.} If a level set $S(0,h)$ contains an equilibrium, then by Lemma~\ref{orderedeq} this equilibrium is unique, and we term it $e(h)$. Define $M^{'}$ to be the supremum of values $h$ such that $S(0, h)$ contains an equilibrium. Clearly $0 < M^{'} \leq M$. By Lemma~\ref{hplusminuslem}, for each $h \in [0, M^{'})$, $S(0, h)$ contains an equilibrium. We thus get a bijective, order preserving, map $e: [0, M^{'}) \to E$.

\vspace{0.25cm}
\begin{lemma}
\label{Ehomeoline}
$e$ is a homeomorphism.
\end{lemma}
\begin{proof}
We already know that $e$ is bijective. The inverse $e^{-1} = \left.H\right|_E$ is continuous as $H$ is continuous. It remains to show that $e$ is continuous. Fix any $h \in [0, M^{'})$ and consider any sequence of values $(h_i) \subset [0, M^{'})$ with $h_i \to h$, and the corresponding equilibria $e_i \equiv e(h_i)$. Let $h_{max} = \sup_i(h_i) < M^{'}$. Then $e_i \leq e(h_{max})$, i.e. $\{e_i\} \subset e(h_{max})^{-}$ which is compact by Lemma~\ref{smalldelta1}. Thus $(e_i)$ contains no divergent subsequences. Consider any convergent subsequence of $(e_i)$, say $e_{i_k} \to \tilde{e}$. By closure of $E$, $\tilde{e} \in E$, and by continuity of $H$, $\tilde{e} \in S(0,h)$. By Lemma~\ref{orderedeq}, this is the only equilibrium on $S(0,h)$, i.e. $\tilde{e} = e(h)$. Thus $e_{i_k} \to e(h)$, and since the subsequence was arbitrary, $e_i \to e(h)$, proving that $e$ is continuous. 
\end{proof}\\

\begin{lemma}
\label{Eunbounded}
$E$ is unbounded. 
\end{lemma}
\begin{proof}
$E$ is closed, but homeomorphic to a half-open interval $[0, M^{'})$. Thus $E$ must be unbounded. 
\end{proof}\\

Incidentally, the claim in Lemma~\ref{Eunbounded} also follows directly from Lemma~\ref{hplusminuslem}, and thus does not require strong monotonicity. Via Lemmas~\ref{Ldef}~to~\ref{Lincreases} we define and explore a scalar function $L$ which serves as a Liapunov function on each level set. 

\vspace{0.25cm}
\begin{lemma}
\label{Ldef}
Given any $y \in Y$, $(y - \partial K) \cap E$ consists of a unique point. 
\end{lemma}
\begin{proof}
Firstly, $0 \in y^{-} \cap E$, so $y^{-} \cap E$ is nonempty; secondly, by Lemma~\ref{smalldelta1}, $y^{-}$ is bounded, and since, by Lemma~\ref{Eunbounded}, $E$ is unbounded, $E \not \subset y^{-}$, i.e. $Y\backslash y^{-} \cap E$ is nonempty. So $E$ intersects both $y^{-}$ and $Y\backslash y^{-}$. By Lemma~\ref{Ehomeoline}, $E$ is connected as the continuous image of connected set. Thus there must be a point in $(y - \partial K) \cap E$ for otherwise $(y -\mathrm{int}(K)) \cap E$ and $(\mathbb{R}^n\backslash (y-K)) \cap E$ are a separation of $E$. Suppose $(y - \partial K) \cap E$ contains two points $p$ and $q$. By Lemma~\ref{orderedeq}, we can choose $p \ll q$. But $p \ll q \leq y$ implies that $p \ll y$ contradicting the fact that $p \in y - \partial K$. 
\end{proof}

\vspace{0.25cm}
\begin{definition}
As a consequence of Lemma~\ref{Ldef}, define $Q:Y \to E$ by $Q(y) = (y - \partial K) \cap E$, and $L: Y \to [0, M^{'})$ by $L(y) = H(Q(y))$. 
\end{definition}
\vspace{0.25cm}

\begin{lemma}
\label{Lmax}
If $S(0,h)\cap E \not = \emptyset$, then $L(y) < L(e(h))$ for all $y \in S(0,h)\backslash\{e(h)\}$.  
\end{lemma}
\begin{proof}
$Q(y) \leq y$ by definition, and since $y \not \in E$, $Q(y) \not = y$, so $Q(y) < y$. Thus $L(y) = H(Q(y)) < H(y) = h = L(e(h))$. 
\end{proof}\\

\begin{lemma}
\label{Lcont}
$L$ is continuous. 
\end{lemma}
\begin{proof}
Since $H$ is continuous, $L$ is continuous provided that $Q$ is. Consider any $y \in Y$, a sequence $y_i \to y$, and the values $Q(y_i) = (y_i - \partial K) \cap E$. Define $x_i = y_i - Q(y_i)$ and note that $x_i \in \partial K$. By Lemma~\ref{upperboundlem}, since $\{y_i\}$ is bounded, we can find $z \in Y$ with $z > \{y_i\}$ and hence $z > \{Q(y_i)\}$. Since $z^{-}$ is bounded by Lemma~\ref{smalldelta1}, $(Q(y_i))$ contains no divergent subsequences. Consider any convergent subsequence of $(Q(y_i))$, say $Q(y_{i_k}) \to q$. Since $E$ is closed, $q \in E$. We have $x_{i_k} \to y-q$. Since  $\{x_{i_k}\} \subset \partial K$, and $\partial K$ is closed, $y-q \in \partial K$, i.e. $q \in y - \partial K$. Since the intersection between $y - \partial K$ and $E$ consists of the unique point $Q(y)$ (Lemma~\ref{Ldef}), $q = Q(y)$. As the subsequence $(y_{i_k})$ was arbitrary, $Q(y_i) \to Q(y)$ proving that $Q$ is continuous. 
\end{proof}\\

\begin{lemma}
\label{Lincreases}
If $y \not \in E$, then $L(\phi_t(y)) > L(y)$ for all $t > 0$. I.e. $L$ increases strictly along nontrivial orbits.
\end{lemma}
\begin{proof}
Consider any $y \not \in E$ and let $h = H(y)$ so that $y \in S(0,h)$. By the definition of $Q(y)$, $Q(y) < y$. Strong monotonicity implies that for any $t > 0$, $\phi_t(Q(y)) = Q(y) \ll \phi_t(y)$. Consider any $e \in E$. If $e \leq Q(y)$, then $e \ll \phi_t(y)$, i.e. $e \not = Q(\phi_t(y))$. So, $Q(\phi_t(y)) > Q(y)$, and thus $L(\phi_t(y)) > L(y)$. 
\end{proof}\\

The main theorem in this paper can now be proved.

{\em Proof of Theorem~\ref{mainthm}}. Lemma~\ref{orderedeq} and the remarks following this lemma establish the existence of $M^{'} \leq M$, such that for each $h \in [0, M^{'})$, $S(0,h)$ contains a unique equilibrium and if $M^{'} \not = M$, then for $h \in [M^{'}, M)$, $S(0,h)$  contains no equilibria. By Lemma~\ref{Lincreases}, a scalar function $L$ increases strictly along nontrivial orbits. Moreover, if $S(0,h)$ contains an equilibrium, then, by Lemma~\ref{Lmax}, $L$ takes a maximum at this equilibrium, ensuring that all orbits on $S(0,h)$ converge to this equilibrium $e(h)$. If $h \in [M^{'}, M)$, then by a standard argument $S(0,h)$ can contain no $\omega$-limit sets. Assume the contrary and assume that there is a nonequilibrium point $z \in S(0, h)$ such that $\phi_{t_k}(y) \to z$ for some $y \in S(0, h)$ and some sequence of times $t_k \to \infty$. By continuity of $L$ (Lemma~\ref{Lcont}), $L(\phi_{t_k}(y)) \to L(z)$, and as $L$ increases along orbits, $L(\phi_{t}(y)) < L(z)$ for all $t \geq 0$. Since $z \not \in E$, $L(\phi_s(z)) > L(z)$ for any $s > 0$. By continuity of the flow, $L(\phi_{t_k + s}(y)) \to L(\phi_s(z)) > L(z)$, contradicting the fact that $L(\phi_{t}(y)) < L(z)$ for all $t \geq 0$. $\square$

\section{An example}
\label{secexample}

The system of two chemical reactions involving three substrates, $A$, $B$ and $C$:
\[
A + B \rightleftharpoons C,  \quad A \rightleftharpoons B,
\]
with no information on the kinetics except a weak monotonicity condition on reaction rates \cite{banajiSIAM}, gives rise to a dynamical system on $\mathbb{R}^3_{\geq 0}$
\begin{equation}
\label{chemsys}
\left. 
\begin{array}{ccl}
\dot x_1 & = & -f_1(x_1, x_2, x_3) - f_2(x_1, x_2) \\
\dot x_2 & = & -f_1(x_1, x_2,x_3) + f_2(x_1, x_2) \\
\dot x_3 & = & f_1(x_1, x_2,x_3)
\end{array}
\qquad
\right\}
\end{equation}
where $x_1, x_2, x_3$ are the concentrations of $A, B, C$ respectively, and $f_1, f_2$ are arbitrary $C^1$ functions satisfying $f_{11} \equiv \frac{\partial f_1}{\partial x_1} \geq 0$, $f_{12} \equiv  \frac{\partial f_1}{\partial x_2} \geq 0$, $f_{13} \equiv  \frac{\partial f_1}{\partial x_3} \leq 0$, $f_{21} \equiv  \frac{\partial f_2}{\partial x_1} \geq 0$, $f_{22} \equiv  \frac{\partial f_2}{\partial x_2} \leq 0$. It is easy to check that the scalar function $H(x_1, x_2, x_3) = x_1 + x_2 + 2x_3$ is preserved by the system. The level sets of this function are termed ``stoichiometric classes'' of the system. 

\begin{theorem}
If we assume that $f_{13} < 0$, $f_{21} > 0$ and $f_{22} < 0$ everywhere on $\mathbb{R}^3_{\geq 0}$, then (\ref{chemsys}) is globally convergent, i.e. each orbit converges to an equilibrium, which is unique on the associated level set of $H$.
\end{theorem}
\begin{proof}
Note first that the assumptions $f_{13} < 0$, $f_{21} > 0$ and $f_{22} < 0$ are satisfied if both reactions are reversible and common kinetics (including, for example, mass-action kinetics) are assumed. It would not be reasonable to assume that $f_{11} > 0$ or $f_{12} > 0$ everywhere on $\mathbb{R}^3_{\geq 0}$: in particular, physical constraints mean that either $x_1 = 0$ or $x_2 = 0$ imply $f_1 = 0$ and hence $f_{11} = f_{12} = 0$. 

Define $Y \equiv \mathbb{R}^3_{\geq 0}$. In \cite{banajidynsys} it was shown that (\ref{chemsys}) preserves a proper cone
\[
K = \{(x_1, x_2, x_3)\in \mathbb{R}^3\,:\, x_3 \geq 0, x_1 + x_3 \geq 0, x_2 + x_3 \geq 0, x_1 +x_2 + x_3 \geq 0\}. 
\]
Certainly any nonnegative vector satisfies these inequalities, so $K \supset Y$. However $Y$ is a proper subset of $K$. Defining $J$ to be the Jacobian of (\ref{chemsys}), $\alpha \equiv f_{11} + f_{12} - f_{13} + f_{21} - f_{22}$, and $J^{'} \equiv J + \alpha I$, direct calculation gives that, with $f_{13} < 0$, $f_{21} > 0$ and $f_{22} < 0$, $J^{'}$ maps each extremal vector of $K$ into the interior of $K$, and hence is $K$-irreducible. By results in \cite{hirschsmithalt}, the flow generated by (\ref{chemsys}) is strongly monotone with respect to the order generated by $K$. 

We can check that $\nabla H = [1,1,2]^T \in \mathrm{int}(K^*)$. First, for $x = [x_1, x_2, x_3]^T \in K$, $\nabla H \cdot x = x_1 + x_2 + 2x_3 = (x_1 +x_2 + x_3) + x_3 \geq 0$. Second, if $x_1 + x_2 + 2x_3 = 0$, then we must have both $x_1 +x_2 + x_3 = 0$ and $x_3 = 0$, implying both that $x_1 + x_2 = 0$ and $x_1, x_2 \geq 0$. Thus $\nabla H \cdot x = 0 \Rightarrow x_1, x_2,x_3 = 0$, confirming that $\nabla H \in \mathrm{int}(K^*)$. 

All level sets are planar and are bounded. Thus by Lemma~\ref{boundedeq}, each level set contains a unique equilibrium, and by Theorem~\ref{mainthm}, all trajectories on a level set converge to this equilibrium. 
\end{proof}\\

\section{Concluding remarks}

We note that at several points in the proof of our main result, a reduction in generality would have considerably simplified the arguments. Most dramatic of all, restricting to linear first integrals would have made it immediate that the portions of level sets termed $S(c, h)$ were topologically balls. Alternatively insisting that the ordering defined by $K$ made $Y$ into a lattice would have allowed a rapid proof of the fact that the equilibrium set was ordered, and again greatly simplified the paper. However, as the example above illustrates, non-simplicial preserved cones which do not induce a lattice ordering on $\mathbb{R}^n$ may arise naturally in applications. 

Some of the results in this paper extend, with only minor modifications, to the case where monotonicity is not necessarily strong. On the other hand, the simple structure of the equilibrium set, key to global convergence, is no longer automatic. We will in future work consider nontrivial extensions removing the requirement of strong monotonicity, and to situations where there may be more than one integral, as arise freqently in applications from chemistry.

\bibliographystyle{siam}

\end{document}